\documentclass[a4paper,11pt]{article} 

\usepackage[utf8]{inputenc} 
\usepackage{amsfonts}
\usepackage{amssymb}
\usepackage{amsthm}
\usepackage{amsmath}
\usepackage{a4wide}
\usepackage{mathrsfs}
\usepackage{mathtools}
\usepackage{epsfig}
\usepackage{palatino}
\usepackage{tikz}
\usepackage{bm}
\usepackage{fp,ifthen}
\usepackage{nicefrac}
\usetikzlibrary{decorations.pathreplacing}
\usepackage{float}
\usepackage{tcolorbox}
\usepackage{enumerate} 
\usepackage{enumitem} 
\usepackage{bbm}
\usepackage[italian,english]{babel}

\usepackage{nicematrix}

\usepackage{nccmath}

\mathtoolsset{showonlyrefs}

\let\OLDthebibliography\thebibliography
\renewcommand\thebibliography[1]{
  \OLDthebibliography{#1}
  \setlength{\parskip}{0pt}
  \setlength{\itemsep}{4pt plus 0.3ex}
}

\newcommand{\mres}{\mathbin{\vrule height 1.6ex depth 0pt width
0.13ex\vrule height 0.13ex depth 0pt width 1.3ex}}

\usepackage[a4paper,twoside,top=2.5cm, bottom=2cm, left=2.5cm, right=2.5cm]{geometry} 

\newcommand{\pdfgraphics}{\ifpdf\DeclareGraphicsExtensions{.pdf,.jpg}\else\fi}
\usepackage{graphicx}
\usepackage{makecell}

\usepackage{lipsum}
\usepackage{color}
\definecolor{hanblue}{rgb}{0.27, 0.42, 0.81}
\definecolor{red}{rgb}{1.0, 0.0, 0.0}
\usepackage[colorlinks, citecolor=blue,linkcolor=blue, urlcolor = blue]{hyperref}

\theoremstyle{plain}

\newtheorem{teo}{Theorem}[section]
\newtheorem{lemma}[teo]{Lemma}
\newtheorem{prop}[teo]{Proposition}
\newtheorem{cor}[teo]{Corollary}

\theoremstyle{definition}

\newtheorem{rem}[teo]{Remark}

\theoremstyle{remark}

\numberwithin{equation}{section}

\renewcommand{\epsilon}{\varepsilon}
\newcommand{\N}{\ensuremath{\mathbb{N}}}
\newcommand{\Z}{\ensuremath{\mathbb Z}}
\newcommand{\R}{\ensuremath{\mathbb R}}

\newcommand{\supp}{\textnormal{supp}}

\newcommand{\link}{\operatorname{link}}

\newcommand{\dist}{\operatorname{dist}}

\usepackage{esint}

\newcommand{\s}{\hspace{7pt}}

\newcommand{\vsp}{\vspace{3pt}}
\newcommand{\Sp}{\mathbb{S}}
\newcommand{\ep}{\varepsilon}
\newcommand{\E}{\mathcal{E}}

\newcommand{\M}{\mathbb{M}}
\newcommand{\D}{\mathcal{D}}
\newcommand{\F}{\mathbf{F}}
\newcommand{\Leb}{\mathscr{L}}
\newcommand{\G}{\mathscr{G}}

\newcommand{\lcurr}{[\hspace{-1.3pt}[}
\newcommand{\rcurr}{]\hspace{-1.3pt}]}

\usepackage{tcolorbox}

\DeclarePairedDelimiter{\abs}{\lvert}{\rvert}

\makeatletter
\renewcommand*\env@matrix[1][*\c@MaxMatrixCols c]{%
  \hskip -\arraycolsep
  \let\@ifnextchar\new@ifnextchar
  \array{#1}}
\makeatother

\usepackage{url}
\usepackage{authblk}

\makeatletter
\newcommand{\addressa}[1]{\gdef\@addressa{#1}}
\newcommand{\emaila}[1]{\gdef\@emaila{\url{#1}}}

\newcommand{\addressb}[1]{\gdef\@addressb{#1}}
\newcommand{\emailb}[1]{\gdef\@emailb{\url{#1}}}

\newcommand{\addressc}[1]{\gdef\@addressc{#1}}
\newcommand{\emailc}[1]{\gdef\@emailc{\url{#1}}}

\newcommand{\@endstuff}{\par\vspace{\baselineskip}\noindent
\begin{tabular}{@{}l}\scshape\@addressa\\\textit{E-mail address:} \@emaila\end{tabular} 

\vspace{12pt} \noindent
\begin{tabular}{@{}l}\scshape\@addressb\\ \textit{E-mail address:} \@emailb\end{tabular}

\vspace{12pt} \noindent
\begin{tabular}{@{}l}\scshape\@addressc\\ \textit{E-mail address:} \@emailc\end{tabular}
}

\AtEndDocument{\@endstuff}
\makeatother

\begin{document}

\pdfgraphics 

\title{Coercivity and Gamma-convergence of the $p$-energy of sphere-valued Sobolev maps}


\author{Michele Caselli, Mattia Freguglia and Nicola Picenni}


\addressa{Michele Caselli \\ Scuola Normale Superiore, Piazza dei Cavalieri 7, 56126 Pisa, Italy}
\emaila{michele.caselli@sns.it}

\addressb{Mattia Freguglia \\ Scuola Normale Superiore, Piazza dei Cavalieri 7, 56126 Pisa, Italy }
\emailb{mattia.freguglia@sns.it}

\addressc{Nicola Picenni \\ University of Pisa, Department of Mathematics \\ Largo Bruno Pontecorvo 5, 56127 Pisa, Italy }
\emailc{nicola.picenni@unipi.it}


\date{}

\maketitle

\vspace{-0.5cm}
\begin{abstract}

\noindent We consider sequences of maps from an $(n+m)$-dimensional domain into the $(n-1)$-sphere, which satisfy a natural $p$-energy growth, as $p$ approaches $n$ from below. We prove that, up to subsequences, the Jacobians of such maps converge in the flat topology to an integral $m$-current, and that the $p$-energy Gamma-converges to the mass of the limit current.

As a corollary, we deduce that the Jacobians of $p$-energy minimizing maps converge to an integral $m$-current that is area-minimizing in a suitable cobordism class, depending on the boundary datum. Moreover, we obtain new estimates for the minimal $p$-energy of maps with prescribed singularities.

\end{abstract}

 \tableofcontents

\vspace{5ex}



\vsp





\section{Introduction}


In this work, we study the asymptotic behavior of sequences of sphere-valued Sobolev maps $u_p \in W^{1,p}_{\rm loc}(\mathbb{R}^{n+m}; \mathbb{S}^{n-1}),$ as $p \in (n-1,n)$ approaches $n$ (which is the first exponent in which $m$-dimensional singularities are not in the energy space), subject to a uniform bound on their rescaled $p$-energies
\begin{equation}\label{eq: first eq uniform bound}
   (n-p) \int |\nabla u_p|^p \, dx
\end{equation}

Our main result, Theorem \ref{main_res}, provides a general compactness theorem for the Jacobians of these maps. Moreover, we prove $\Gamma$-convergence of the $p$-energy, properly rescaled as in \eqref{eq: first eq uniform bound}, to the mass of the limit current of the Jacobians. Our main theorem is in complete analogy with the results for the Ginzburg-Landau energy in general codimension obtained by Alberti, Baldo, and Orlandi in \cite{2005-Indiana-ABO}, which extended the earlier work in codimension two by Jerrard and Soner \cite{Jerrard-Soner-calcvar}. 

As a corollary of our analysis in the case of fixed boundary conditions (Section \ref{sec: boundary val problems}), we recover several previous results, including one by Hardt and Lin on the convergence of the energy densities of $p$-energy minimizing maps with fixed boundary conditions, as $p$ approaches~$n$. We also obtain an ``oriented version" of this convergence for the Jacobians of these maps.

The proof of the compactness and $\Gamma$-liminf inequality takes up most of the work, and is divided into two macro-steps. First, we address the case of dimension zero $m=0$, that is when the model singularities are isolated points. Then, we extend the result to general dimension and codimension using the zero-dimensional case. While the extension to arbitrary dimensions (Section \ref{sec: compactness and liminf general dim}) follows the techniques of deformation onto grids developed in \cite{2005-Indiana-ABO} (which in turn were inspired by the deformation theorem of \cite{1960-FF}), the core novelty of our work lies in the zero-dimensional setting (see Theorem \ref{teo:X_M+X_F}), where we provide a direct and elementary, yet highly nontrivial, approach to the compactness and liminf inequality. 

Our approach is completely variational and applies to general sequences satisfying the natural energy growth \eqref{eq: first eq uniform bound} without minimality or stationarity assumptions. For this reason, we believe that the techniques we develop in the proof of Theorem~\ref{teo:X_M+X_F} may apply to a broader class of problems involving singular limits.

\subsection{Main result}

Let $\omega_{n-1}:=\mathcal{H}^{n-1}(\Sp^{n-1})$ be the volume of the unit sphere in $\R^{n}$ and let us set also $\gamma_n:=\omega_{n-1}/n$. Our main result is the following.

\begin{teo}\label{main_res}
Let $n\geq 2$ and $m\geq 0$ be integers, and let $\Omega\subset \R^{n+m}$ be a bounded open set with Lipschitz boundary. For a (discrete) sequence $p\in (n-1,n)$ let $u_p \in W^{1,p}(\Omega;\Sp^{n-1})$ be maps such that
\begin{equation}\label{hp:energy-bound}
\limsup_{p \to  n^-} \, (n-p)\int_{\Omega} |\nabla u_p |^{p}\,dx<+\infty.
\end{equation}

\noindent Then the following statements hold.
\begin{itemize}
    \item[$(i)$](Compactness and $\Gamma$-liminf inequality) There exists an integral $m$-boundary $\Sigma$ in $\Omega$ such that, up to a subsequence, the Jacobians $\{\star Ju_{p}\}$ converge to $\gamma_n \Sigma$ in the flat topology of $\Omega$ and (along this subsequence) it holds that
    \begin{equation}\label{eq: liminf main}
        \liminf_{p \to  n^-} \, (n-p)\int_{\Omega} |\nabla u_p |^{p}\,dx \geq (n-1)^{\frac{n}{2}}\omega_{n-1} \M_\Omega(\Sigma).
    \end{equation}
    
    \item[$(ii)$] ($\Gamma$-limsup inequality) For every integral $m$-boundary $\Sigma$ on $\Omega$ and every $p\in (n-1,n)$ there exists maps $u_p\in W^{1,p}(\Omega;\Sp^{n-1})$ such that $\star Ju_p\to \gamma_n\Sigma$ as $p\to n^-$ in the flat topology of $\Omega$ and 
$$ \limsup_{p \to  n^-} \, (n-p)\int_{\Omega} |\nabla u_p |^{p}\,dx \leq (n-1)^{\frac{n}{2}}\omega_{n-1} \M_\Omega(\Sigma).$$
\end{itemize}
\end{teo}

The following result is an immediate corollary of the previous theorem. 

\begin{cor}\label{cor: correnti compactness}
Let $n\geq 2$ and $m\geq 0$ be integers, and let $\Omega\subset \R^{n+m}$ be a bounded open set with Lipschitz boundary. For a (discrete) sequence $p\in (n-1,n)$, let $\Sigma_p$ be an integral $m$-boundary and set
\begin{equation}\label{min_prescr_J}
   \E^{\rm  min}_p(\Sigma_p, \Omega) := \inf \biggl\{ \int_{\Omega} |\nabla v|^p \, : \, v\in W^{1,p}(\Omega;\Sp^{n-1}),\ \star Jv= \gamma_n \Sigma_p \biggr\} .
\end{equation}
If there holds
\begin{equation*}
    \limsup_{p \to n^-} \, (n-p) \E^{\rm  min}_p(\Sigma_p, \Omega) <+\infty , 
\end{equation*}
then there exists an integral $m$-boundary $\Sigma$ in $\Omega$ such that, up to a subsequence, $\Sigma_p\to \Sigma$ in the flat topology of $\Omega$ and (along this subsequence) it holds that
$$\liminf_{p \to  n^-} \,  (n-p) \E^{\rm  min}_p(\Sigma_p, \Omega) \geq (n-1)^{\frac{n}{2}}\omega_{n-1} \M_\Omega(\Sigma) . $$
\end{cor}

With some technical adjustments, we also obtain a variant of Theorem~\ref{main_res} with Dirichlet boundary conditions (see Theorem~\ref{thm:bvp}). Since the complete statement is quite technical and would require a lot of terminology to state, we do not include it here. Nevertheless, we anticipate that, as an immediate corollary, we obtain that the Jacobians of $p$-energy minimizing maps with a fixed boundary datum $g\in W^{1-1/n,n}(\partial \Omega;\Sp^{n-1})$ converge to an area-minimizing integral $m$-current with prescribed boundary equal to the Jacobian of $g$; see Corollary~\ref{cor: Hardt-Lin for currents}.

\begin{rem}\label{rem:Millot}
In view of possible applications of our result to the study of the convergence of critical points or in the context of min-max (cf. \cite{PPS-CPAM-Gammaconv}), we remark that Theorem~\ref{main_res} still holds if one considers a variant of the $p$-energy in which $|\nabla u|^p$ is replaced by $(1+|\nabla u|^2)^{p/2}$. Indeed, the compactness and liminf part of the statement follows trivially because the modified functional is larger than the original one, while the limsup part only requires very minor modifications in the computations.

This variant, while being completely analogous to the $p$-energy from the variational point of view, has the advantage that its Euler-Lagrange equation is uniformly elliptic, so it avoids the degeneracy of the $p$-Laplace operator.
\end{rem}

\begin{rem}\label{rem:manifolds}
Our results Theorem \ref{main_res} and Corollary \ref{cor: correnti compactness} should also hold in the case of an ambient $(n+m)$-dimensional compact Riemaniann manifold $M$ (possibily with boundary). The statements regarding the compactness and liminf inequality can be easily deduced by a localization argument, while the construction of the recovery sequence requires some more care (cf. \cite[Remark~(v) after Theorem~1.1]{2005-Indiana-ABO}). 
\end{rem}

\begin{rem}
We point out that in our setting, we cannot expect to have compactness for the sequence of maps $\{u_p\}$ in any stronger sense than the weak* convergence in $L^\infty$, which follows trivially from the fact that our maps are sphere-valued. Indeed, it is not difficult to construct sequences with bounded rescaled $p$-energy that converge to the null map weakly* in $L^\infty$, and this means that we cannot have compactness in any topology that implies pointwise convergence.

In particular, the limit current $\Sigma$ provided by the compactness part of Theorem~\ref{main_res} cannot be seen as the Jacobian, nor the singular set, of a limit map. This phenomenon, which is typical of the $\Gamma$-convergence setting (cf. \cite[Remark~(vi) after Corollary~1.2]{2005-Indiana-ABO}), is a striking difference with respect to the behavior of minimizers (see \cite{HardtLinChangyou}) or stationary points (see \cite{Stern2018}).
\end{rem}

\subsection{Motivation and related works}


\subsubsection{The Ginzburg-Landau and Yang-Mills-Higgs frameworks}

The readers already familiar with the $\Gamma$-convergence theory in similar frameworks will notice that our result Theorem \ref{main_res} is formally analogous to the results in \cite{Jerrard-Soner-calcvar, 2005-Indiana-ABO, PPS-CPAM-Gammaconv,PPS-Gammaconv-codim3}. In particular, in the Ginzburg–Landau framework of \cite{Jerrard-Soner-calcvar, 2005-Indiana-ABO} for the energies
\begin{equation*}
     G_\ep(u) := \int \frac{1}{n}|\nabla u |^n + \frac{(1-|u|^2)^2}{4\ep^2} , 
\end{equation*}
and in the Yang-Mills-Higgs framework \cite{PPS-CPAM-Gammaconv}, in codimension two, for the energy
\begin{equation*}
      E_\ep(u, \nabla) := \int  |\nabla u |^2 + \ep^2 |F_\nabla|^2 +  \frac{(1-|u|^2)^2}{4\ep^2} 
      .
\end{equation*}

A key feature that distinguishes our setting from the $\Gamma$-convergence results in these frameworks is the relationship between energy bounds and the mass of the Jacobians. In the Yang-Mills-Higgs setting studied in \cite{PPS-CPAM-Gammaconv}, the gauge-invariant Jacobian $J(u,\nabla)$ is controlled pointwise by the energy density, that is
\begin{equation*}
    |J(u,\nabla)| \le |\nabla u |^2 + \ep^2 |F_\nabla|^2 +  \frac{(1-|u|^2)^2}{4\ep^2} . 
\end{equation*}
The gauge-invariant Jacobian is a $2$-form that plays a role analogous to the standard Jacobian $2$-form $Ju := du^1 \wedge du^2$ in the Ginzburg-Landau setting; we refer to \cite[Eq. (1.3)]{PPS-CPAM-Gammaconv} for the definition of $J(u,\nabla)$. Hence, in this setting, a uniform bound on the energies $E_\ep(u_\ep, \nabla_\ep) \le C $ as in \cite[Theorem 1.2]{PPS-CPAM-Gammaconv} automatically implies a uniform bound on the masses of the Jacobians $\|J(u_\ep,\nabla_\ep)\|_{L^1}$. As a consequence, the main challenge to be dealt with in the Yang-Mills-Higgs setting is not the extraction of a convergent subsequence from the Jacobians but to prove that the limit $(n-2)$-dimensional current is integral. 

On the other hand, in the classical work \cite{2005-Indiana-ABO}, by the pointwise estimate for the Jacobian
\begin{equation*}
    |Ju| \le |\nabla u|^n ,
\end{equation*}
the energy bound in the vorticity regime $G_\ep(u_\ep) \le C\abs{\log(\ep)}$ still implies a logarithmic control on the mass of the Jacobians $\|Ju_\ep\|_{L^1}$. Then, the proof of $\Gamma$-convergence relies on a delicate and highly nontrivial application of the degree estimates in \cite{Sandier-JFA-1998,1999-SIAM-Jerrard}. 

Our setting presents a departure from both these scenarios: the energy bound \eqref{hp:energy-bound} provides no immediate control on the mass of the Jacobians $\star Ju_p$, that here are not even $L^1$ functions, but are supported on $m$-dimensional sets and might have infinite mass even if the $p$-energy is finite. This feature represents a fundamental difference from the diffuse frameworks above and requires new ideas to establish compactness.

\subsubsection{The \texorpdfstring{$p$-energy}{p-energy} framework} 

The study of $p$-energy minimizing maps into spheres as $p$ approaches the critical Sobolev exponent was initiated by Hardt and Lin \cite{HL-CalcVar}, who analyzed sequences of energy minimizers $u_p: \Omega \to \Sp^1$ from planar domains as $p \to 2^-$, proving that they converge smoothly away from exactly $|d|$ singular points, where $d$ is the degree of $u_p |_{\partial \Omega}$. This was extended to higher dimensions in \cite{HardtLinChangyou}, where it is shown that, for $\Omega \subset \R^n$ and $p\to n^-$, energy minimizing maps $u_p: \Omega \to \Sp^{n-1}$ converge away from $|d|$ singularities that minimize a certain renormalized energy. 

From here, it is natural to ask what happens to the singular set of a sequence of $p$-energy minimizing maps in general dimension and codimension. Even though this question is extremely natural and, for example, its analogous in the Ginzburg-Landau framework is well-known \cite[Theorem I.1]{Riv-Lin-GL}, in the case of $p$-energy minimizing maps with general hypotheses it has only been addressed in an unpublished (but highly nontrivial) work of Hardt and Lin, which seems to appear in literature only in \cite{Lin-2011}. 

This work by Hardt and Lin concerns the asymptotic behavior of the energy densities of $p$-energy minimizing maps from an $(n+m)$-dimensional domain $\Omega \subset \R^{n+m}$ into $\Sp^{n-1}$, with fixed boundary conditions $g \in W^{1,n-1}(\partial \Omega; \Sp^{n-1})$. Observe that for $g \in W^{1,n-1}(\partial \Omega; \Sp^{n-1})$ its (tangential) Jacobian $\star J g $ is a well-defined integral current supported on $\partial \Omega$. With our notation, the result in \cite{Lin-2011} reads as follows.

\begin{teo}[Theorem 3.1 in \cite{Lin-2011}]\label{thm: Hardt-Lin unpublished} Let $n\ge 2$, $m\ge 0$ be integers, $p\in (n,1-n)$ and $\mu_p$ be the normalized energy densities 
\begin{equation}\label{eq: norm energy density}
    \mu_p := \frac{n-p}{\omega_n(n-1)^{p/2}} |\nabla u_p(x)|^p dx\mres \Omega 
\end{equation}
associated with maps $u_p \in W^{1,p}(\Omega;\Sp^{n-1})$ that minimize the $p$-energy in $\Omega$ with fixed boundary datum $g \in W^{1,n-1}(\partial \Omega; \Sp^{n-1})$.

Then, as $p \to  n^-$, a subsequence of $\mu_p$ weakly converges to a Radon measure $\mu$ such that $\supp(\mu) = \supp(\Sigma)$ and $\mu(\Omega) = \M_\Omega(\Sigma)$, where $\Sigma$ is an integral area-minimzing current in $\Omega$ with $\partial \Sigma = \star J g$.
\end{teo}

Actually, Theorem \ref{thm: Hardt-Lin unpublished} holds only if the $n$-th homology group of $\overline{\Omega}$ is trivial, otherwise $\Sigma$ is area-minimizing only in a suitable cobordism class depending on $g$ (cf. Section \ref{sec: boundary val problems} below or \cite[Theorem 5.5]{2005-Indiana-ABO}). 

A significant shift in generality in this area came with the work of Stern \cite{Stern2018}, who removed the minimality assumption and studied general sequences of stationary $p$-harmonic maps $u\in W^{1,p}(M; \Sp^1)$, in an ambient Riemannian manifold $(M,g)$, satisfying only the natural energy growth condition $\| \mu_p\|(M) \le C$ as $p\to 2^-$ (where $\mu_p$ is defined as in Theorem \ref{thm: Hardt-Lin unpublished}). Here, Stern proved that the singular sets converge to the support of a stationary rectifiable $(n-2)$-varifold $V$ with density bounded below. Moreover, in dimension $n=2$, he also shows that the density of $\|V\|$ is quantized, taking values in $2\pi \N$. This result mirrors the integrality of point vortices in the Ginzburg-Landau framework but is achieved through a distinct blow-up analysis tailored to $p$-harmonic maps.

More recently, in \cite{VS-VV-CalcVar-23}, the authors studied the behavior of minimizing $p$-harmonic maps from two-dimensional domains, with values in general compact Riemannian manifolds, as $p$ approaches $2^-$. In \cite{BVSVV}, they also considered minimizing maps defined on three-dimensional domains but under the additional assumption that the fundamental group of the target manifold is finite. These works generalize the target space and obtain several interesting results but are restricted to minimizing maps and low-dimensional domains.

Our Theorem~\ref{main_res} provides a new perspective on this line of work. While previous results focused either on minimizers \cite{HL-CalcVar, HardtLinChangyou, Lin-2011, VS-VV-CalcVar-23, BVSVV} or stationary points \cite{Stern2018}, we establish flat compactness of Jacobians and $\Gamma$-convergence of the energies for general sequences only satisfying the energy growth condition \eqref{hp:energy-bound}, and we obtain that any limit of the Jacobians is integral. 

As a corollary of our work, since we also establish a $\Gamma$-convergence result with boundary conditions (see Theorem~\ref{thm:bvp}), we deduce both a new proof of Theorem \ref{thm: Hardt-Lin unpublished} by Hardt and Lin, and also an ``oriented version" of it in Corollary~\ref{cor: Hardt-Lin for currents}. We stress again that our proof is completely variational and avoids relying on the first variation of the energy.

Finally, as a further byproduct of our analysis, in Corollary~\ref{cor: correnti compactness} we obtain asymptotic estimates for the minimal $p$-energy of maps with prescribed singularities in any dimension and codimension. This problem was previously studied in \cite{Hardt-Chen-prescribed} and \cite[Section~5]{HardtLinChangyou} in the case in which $m=0$ and the singularities do not depend on $p$ (see also \cite[Sections~1.8,~4.5.2 and~13.6.2]{Bre-Mi}). We remark that here, in the case $m=0$, we also obtain some estimates for the problem \eqref{min_prescr_J} when $p\in (p_0,n)$ for some $p_0\in (n-1,n)$ that is independent of the prescribed Jacobian (see Remark~\ref{rem:fixed_p}), while in \cite{Hardt-Chen-prescribed,HardtLinChangyou} all the energy estimates are valid only when $p$ is larger than some $p_0$ that strongly depends on the prescribed singular sets, for example through the minimal distance between two different singular points.

\vspace{-5pt}

\paragraph{Authors' note} After the first version of this work was completed and shared with some experts, we were informed by Daniel~Stern that, for maps between compact manifolds without boundary, a different proof of the liminf and compactness part of Theorem~\ref{main_res} is contained in Chapter~12 of his PhD thesis \cite{Stern:PhD}. Apparently, his result has gone unnoticed in the subsequent literature. The general strategy of the proof in \cite{Stern:PhD} presents some similarities with ours. Nevertheless, the key ingredients of the $0$-dimensional estimate and the approximation of maps with ones with controlled Jacobians are completely different. 

\subsubsection{The fractional framework}

We mention that recently, following a suggestion in \cite{SerraSurv}, the authors of the present work proposed in \cite{codim2-frac} an approximation framework for the area in codimension two in building on the theory of fractional Sobolev energies. Precisely, for $s\in (0,1)$, in \cite{codim2-frac} we introduced a notion of fractional $s$-mass $\M_s$ for $(n-2)$-dimensional closed boundaries, defined via minimization of the $W^{(1+s)/2,2}$-energy among maps that link (in a suitable sense) the surface. We proved that this fractional $s$-mass (appropriately rescaled by a factor depending on $s$) $\Gamma$-converges, with respect to the flat topology of boundaries, to the standard $(n-2)$-dimensional Hausdorff measure with multiplicity. This provides an approach to the approximation of the area in codimension two in the spirit of the fractional perimeter that, we believe, may offer advantages in regularity and compactness properties compared to classical approximations.

\subsection{Organization of the paper}

The paper is organized as follows. In Section \ref{sec: sec def and notations}, we recall key notions regarding currents and Jacobians of Sobolev maps that we need for our analysis. Section \ref{sec: approximation Jac} provides some approximation results for Sobolev maps with quantitative estimates that will be used crucially in the subsequent sections. Section \ref{sec: zero dim case} is devoted to the proof of the zero-dimensional case, where we establish our Theorem \ref{main_res} for $m=0$ with an elementary, yet nontrivial, proof. In Section \ref{sec: compactness and liminf general dim}, we extend our analysis to general dimension and codimension, proving the compactness and lower bound (Section \ref{sbs: compact and lower bound}) as well as the upper bound (Section \ref{sbs: upper bound}) parts of Theorem \ref{main_res}. Then, we address boundary value problems in Section \ref{sec: boundary val problems}. Lastly, Section~\ref{sec:appendix} is an appendix which contains the proof of two lemmas.

\section{Definitions and notation}\label{sec: sec def and notations}

In this section, we recall some notions concerning currents and Jacobians of Sobolev maps that are fundamental in this paper. We refer to \cite{1960-FF,Federer-book,ABO2-singularities,2005-Indiana-ABO,Bre-Mi} for a more comprehensive treatment of these notions.

\subsection{Currents and flat norm}

Let us briefly recall some notations and basic definitions about currents. Let $\Omega \subset \R^n$ be a general open set and, for $k \in \{0, \dotsc, n\}$, let $\mathcal{D}^k(\Omega)$ denote the space of smooth $k$-forms with compact support in $\Omega$. 

A $k$-dimensional \emph{current} in $\Omega$ is a continuous linear functional on $\D^k(\Omega)$. The \textit{boundary} of a $k$-current $T$ is the $(k-1)$-current $\partial T$ defined by $ \langle \partial T,\omega \rangle = \langle T, d\omega \rangle$ for all $\omega \in \mathcal{D}^{k-1}(\Omega)$, where $d\omega$ is the exterior differential of the form $\omega$. For a $k$-current $T$, the \textit{mass} of $T$ in an open set $A$, denoted by $\M_A(T)$, is defined as
\begin{equation*}
    \M_A(T) := \sup \big\{ T(\omega ) \, : \, \omega \in \D^k(A) \,, |\omega(x)| \le 1 \s \forall \, x\in A \big    \} . 
\end{equation*}

If $T$ has (locally) finite mass, it can be represented  as a (locally) bounded Borel measure with values in $\bigwedge_k(\R^{n})$; in this case we denote by $|T|$ the variation of the measure $T$. Note that, with this notation, $\M_A(T) = |T|(A)$. 

A $k$-current $T$ is called \emph{rectifiable} if it can be represented by integration on a $k$-rectifiable set, up to an integer multiplicity, namely if
$$\langle T,\omega \rangle =\int_{M} \theta(x) \, \omega_x(\tau(x)) \,d\mathcal{H}^k(x) \qquad\forall \omega \in \D^k(\Omega),$$
where $M$ is a $k$-rectifiable set, $\theta:M\to \Z\setminus\{0\}$ is an integer-valued measurable function, and $\tau(x)$ is a measurable choice of a unitary $k$-vector spanning $T_x M$ at $\mathcal{H}^k$-almost every $x\in M$. In this case we write $T=[M,\theta,\tau]$. We observe that, in the case of a rectifiable current $T=[M,\theta,\tau]$, its mass can be expressed as
$$\M_A(T)=\int_{M \cap A} |\theta(x)|\,d\mathcal{H}^k(x).$$

A $k$-current $T$ is said \emph{integral} if both $T$ and $\partial T$ are rectifiable. If $T=\partial S$ for some integral $(k+1)$-current $S$, then we call $T$ an integral $k$-boundary. We point out that integral $k$-boundaries do not necessarily represent trivial elements of the $k$-homology of $\Omega$ since the current $S$ might not be compactly supported in $\Omega$.

If $\Sigma$ is a smooth oriented $k$-submanifold of $\Omega$, it naturally induces the integral current 
\begin{equation*}
    \lcurr \Sigma \rcurr:= [\Sigma,1,\tau_\Sigma] , 
\end{equation*}
where $\tau_\Sigma$ is the orientation of $\Sigma$. Sometimes, if there is no risk of confusion, we just write $\Sigma$ instead of $\lcurr \Sigma \rcurr$, so we denote with the same symbol the submanifold and the associated integral current.

For an integral $k$-boundary $T$ we define the \emph{flat norm} of $T$ as
\begin{equation*}
    \F_\Omega (T):=\inf\{\M_\Omega (S): \text{$S$ is any $(k+1)$-current with $\partial S = T $ in $\Omega$}\}.
\end{equation*}

This can be expressed equivalently by duality as  
\begin{equation*}
    \F_\Omega(T):=\sup\{\left< T,\psi\right>: \psi \in \D^k (\Omega), \, \|d \psi\|_{L^\infty(\Omega)}\leq 1) \} . 
\end{equation*}

\begin{rem} 
We point out that in our previous work \cite{codim2-frac} we used a slightly different definition of the flat norm for integral $k$-boundaries, that is
\begin{equation*}
\F_\Omega^{\rm \, I}(T):=\inf\{\M_\Omega (S): \text{$S$ is an integral $(k+1)$-current with $\partial S =T$ in $\Omega$}\}.
\end{equation*}

Actually, this notion is equivalent to the previous one when $k\in \{0,n-2,n-1\}$, while it might differ in the other cases (see Proposition~4.1 and Proposition~4.2 in \cite{Bre-Mi}).

On the other hand, to our knowledge, it is an open problem to determine whether these two notions induce the same topology on integral currents in the cases in which they do not coincide. Anyway, in the codimension two case considered in \cite{codim2-frac}, that is $k=n-2$, these two notions are the same.
\end{rem}

If $n\geq 2$, $k_1,k_2\geq 1$ are integers with $k_1+k_2=n-1$, and $\Sigma_1, \Sigma_2$ are disjoint, smooth, closed, oriented connected submanifolds of $\R^n$, with dimensions $k_1$ and $k_2$ respectively, we denote by $\link(\Sigma_1, \Sigma_2) \in \Z $ the \emph{linking number} between $\Sigma_1$ and $\Sigma_2$. This notion can be defined in a few different, but equivalent, ways, and we refer to \cite[Section 2.1]{codim2-frac} for the details on different definitions and their equivalence. For example, if $N_1 \subset \R^n$ is a $(k_1+1)$-dimensional submanifold with $\partial N_1 =\Sigma_1$ intersecting $\Sigma_2$ transversally, the linking number can be defined as
$$\link (\Sigma_1,\Sigma_2):=\sum_{x \in N_1\cap \Sigma_2} s(x),$$
where $s(x)=1$ if the orientation on $T_x N_1 \oplus T_x \Sigma_2=\R^n$ induced by the orientations on $N_1$ and $\Sigma_2$ agrees with the orientation of $\R^n$, and $s(x)=-1$ otherwise.

\subsection{Jacobians}

From now on, unless otherwise stated, $\Omega$ will be a bounded Lipschitz domain in $\R^{n+m}$, with $n\geq 2$ and $m\geq 0$, and we will deal with functions defined on such domains and with values in $\R^n$ or, more often, in $\Sp^{n-1}$. We recall that for a Sobolev map $u:\Omega\to \R^n$ taking values in $\Sp^{n-1}$ just means that $|u|=1$ almost everywhere.

In this setting, for $u\in W^{1,n-1}(\Omega;\Sp^{n-1})$ we set
$$\widehat{du_i}:=du_{1} \wedge \dots \wedge du_{i-1} \wedge du_{i+1} \wedge\dots\wedge du_{n},$$
and we consider the following $(n-1)$-form
$$ju:=\sum_{i=1}^{n} (-1)^{i-1} u_i \widehat{du_i}.$$

We observe that, since $u\in W^{1,n-1}(\Omega;\Sp^{n-1})$, the forms $ju$ are in $L^1$, so we can define the following distributional $n$-form
$$Ju:=\frac{1}{n} d(ju),$$
which is called the \emph{Jacobian} of $u$. 

We can also associate to $Ju$ the $m$-current $\star Ju$, which acts on forms in the following way
\begin{equation}\label{def: star J def}
    \langle \star Ju , \psi \rangle := \frac{(-1)^{m+1}}{n} \int_{\Omega} d\psi \wedge ju ,  \qquad \forall \psi \in \D_{m}(\Omega).
\end{equation}

It turns out that the Jacobian of maps in $W^{1,n-1}(\Omega,\Sp^{n-1})$ always belong to some special class of $m$-currents, which are the boundaries of integer rectifiable currents, as the following result shows.

\begin{prop}[Theorem~3.8 in \cite{ABO2-singularities} and Theorem~4.9 and Theorem~4.10 in \cite{Bre-Mi}]\label{prop:Jacobian_integral}
Let $\Omega\subseteq \R^{n+m}$ be an open set and let $u\in W^{1,n-1}(\Omega;\Sp^{n-1})$ be  map. Let us set $\gamma_n:=\omega_{n-1}/n$. Then
$$\star Ju=\gamma_n \partial N,$$
for some integer rectifiable $(m+1)$-current $N$. 

If in addition $u\in C^0(\Omega\setminus \Sigma)$, for some smooth connected oriented $m$-dimensional submanifold $\Sigma$ without boundary (in $\Omega$), then $\partial N=d  \lcurr \Sigma \rcurr$ for some $d\in\N$ that equals the degree of $u$ on $(n-1)$-spheres $S$ with $\link(S,\Sigma)=1$.
\end{prop}

The next result shows some continuity estimates for the Jacobian, as an operator from some functional spaces into $m$-currents, endowed with the flat norm. Since we have been able to find only particular cases of this statement in the literature (that is, when either $m=0$ or $n=2$), we include in the appendix the proof, which anyway is completely analogous to the ones for those special cases.

\begin{teo}[see Theorem~1 in \cite{BN-Invent} and Corollary~1.7 in \cite{Bre-Mi}]\label{teo:BN-Invent}
Let $\Omega\subset \R^{n+m}$ be a bounded Lipschitz domain, let $p\in [n-1,+\infty]$ and let $q\in [1,+\infty]$ be such that $(n-1)/p + 1/q=1$. Then there exists a constant $C_{n}>0$ such that for every $u,v\in W^{1,p}(\Omega;\R^n)\cap L^{q}(\Omega;\R^n)$ it holds that
\begin{equation}\label{th:BN_invent}
\F_\Omega(\star Ju-\star Jv)\leq C_{n} \|u-v\|_{L^q}(\|\nabla u\|_{L^p} + \|\nabla v\|_{L^p})^{n-1} .
\end{equation}
\end{teo}

\section{Quantitative approximation of Sobolev maps and their Jacobians}\label{sec: approximation Jac}

In this section, we prove some approximation results for Sobolev maps with quantitative error estimates. Actually, we follow quite closely the strategy of \cite[Chapter~10]{Bre-Mi}, which in turn is based on the so-called \emph{projection-averaging technique} introduced in \cite[Section~6]{HL-CPAM} (see also \cite{BZ-density,Bethuel-Acta}). However, since we need more precise estimates, we need to be slightly more careful in the choice of parameters and to compute more explicitly some rates of convergence.

Before entering into details, let us introduce some notation that we use in the rest of the paper. In many proofs, we use the letter $C$ to denote possibly different constants, but we specify each time which parameter these constants depend on.

We denote by $B_r ^d$ the ball in $\R^d$ centered in the origin and with radius equal to $r>0$. Let us fix a non-negative radial function $\rho\in C^\infty _c (\R^{n+m};[0,+\infty))$ supported in $B^{n+m} _1$, and with $\int\rho =1$. Let us set
\begin{equation}\label{defn:rho_ep}
\rho_\ep(x):=\frac{1}{\ep^{n+m}}\rho\left(\frac{x}{\ep}\right).
\end{equation}

We recall some simple estimates for the $L^p$ distance between a Sobolev function $u$ and its regularization $u*\rho_\ep$. The following Lemma is essentially \cite[Lemma~15.50]{Bre-Mi}, when $s\in (0,1)$, while it is classical in the case $s=1$ (it follows, for example, from \cite[Proposition~9.3]{Brezis-FAbook}).

\begin{lemma}\label{lemma:sobolev-convolution}
Let $\Omega\Subset U \subset \R^{n+m}$ be bounded Lipschitz domains and let $s\in (0,1)$ and $p \geq 1$ be real numbers. Let $u\in W^{s,p}(U;\R^n)$ be a map, and let us set $u_\ep:=u*\rho_\ep$ for every $\ep<\dist(\Omega , \partial U)$. Then it holds that
$$\|u_\ep - u\|_{L^p(\Omega)} \leq \|\rho\|_{L^\infty}\cdot \ep^s [u]_{W^{s,p}(U)}.$$

If, instead, $u\in W^{1,p}(U;\R^n)$, then
$$\|u_\ep - u\|_{L^p(\Omega)} \leq  \ep \|\nabla u\|_{L^{p}(U)}.$$
\end{lemma}

We are now ready to state and prove the main result of this section, which allows to replace a general sequence of maps with bounded rescaled $p$-energy with more regular maps, satisfying some quantitative estimates. We remark that if $m=0$, then a $0$-submanifold is just a finite collection of points.

\begin{prop}\label{prop:approx}
Let $\Omega\Subset U\subset \R^{n+m}$ be bounded Lipschitz domains, and let $\{p_k\}\subset (n-1,n)$ be a sequence of real numbers such that $p_k\to n^-$. For every $k\in \N$, let $u_k\in W^{1,p_k}(U;\Sp^{n-1})$ be a map, and let us assume that there exists a constant $E>0$ such that
$$(n-p_k)\int_{U} |\nabla u_k|^{p_k}\,dx \leq E \qquad \forall k\in\N.$$

Let also $\delta\in (0,1/8)$ be a real number, and for every $k$ sufficiently large let $\tilde{p}_k\in (n-1,p_k)$ and $q_k>n$ be defined by
$$\tilde{p}_k:=p_k-\delta(n-p_k), \qquad \frac{n-1}{\tilde{p}_k}+\frac{1}{q_k}=1.$$

Then there exist constants $C_1=C_1(n,m,\Omega,U,E)>0$ and $C_2=C_2(n,m,\Omega,U,\delta,E)>0$, such that for every $k$ sufficiently large there exist a smooth oriented $m$-submanifold $\Sigma_k\subset \Omega$ and a map $\tilde{u}_k\in W^{1,\tilde{p}_k}(\Omega;\Sp^{n-1})\cap C^\infty(\Omega\setminus \Sigma;\Sp^{n-1})$ for which $\star J\tilde{u}_k=\gamma_n \lcurr \Sigma_k \rcurr$ and the following estimates hold.

\begin{gather}
\| u_k - \tilde{u}_k\|_{L^{q_k}(\Omega)} \leq C_1 (n-p_k)^{\frac{1}{n-p_k}},\label{est:approx1}
\\
\M_{\Omega}(\star J\tilde{u}_k)=\gamma_n \mathcal{H}^{m}(\Sigma_k)\leq \frac{C_2}{(n-\tilde{p}_k)^{1+3/\delta}}.\label{est:approx2}
\end{gather}

Moreover, there exists a real number $L_\delta>1$ such that $L_\delta\to 1$ as $\delta\to 0^+$ and for every pair of open sets $\Omega_1\subseteq \Omega$ and $\Omega_2\subseteq U$ with $\Omega_1\Subset\Omega_2$, it holds that
\begin{equation}\label{est:approx3}
\limsup_{k\to +\infty} \biggl[ (n-\tilde{p}_k) \int_{\Omega_1} |\nabla \tilde{u}_k|^{\tilde{p}_k} \,dx - L_{\delta} (n-p_k)\int_{\Omega_2} |\nabla u_k|^{p_k}\,dx\biggr]\leq 0.
\end{equation}
\end{prop}

\begin{proof}
For every $a\in B^n _\delta$, let us consider the map $\pi_a: B^n _1 \setminus \{a\} \to \Sp^{n-1}$, that is the radial projection from $a$ onto $\Sp^{n-1}$, namely
\begin{equation}\label{defn:pi_a}
\pi_a(x):=a+t_a(x)\frac{x-a}{\abs{x-a}},
\end{equation}
where
$$t_a(x):=-\frac{\langle a,x-a\rangle}{|x-a|} + \sqrt{\biggl(\frac{\langle a,x-a\rangle}{|x-a|}\biggr)^2+1-|a|^2}$$
is the unique positive number such that the right-hand side of \eqref{defn:pi_a} belongs to $\Sp^{n-1}$.

We observe that, since $|a|<\delta<1/8$, there exist a constant $C>0$, independent of $a$ and $\delta$, such that
\begin{equation}\label{est:t_a}
1-|a|\leq t_a(x)\leq 1+|a|, \qquad\text{and}\qquad
|\nabla t_a(x)| \leq \frac{C|a|}{|x-a|}, \qquad \forall x \in B^n _1\setminus \{a\} , 
\end{equation}
and therefore
\begin{equation}\label{est:pi_a}
|\nabla \pi_a(x)|\leq \frac{C}{|x-a|} \quad \text{and}\quad |\pi_a(x)-y|\leq C|x-y| \qquad \forall x \in B^n_1\setminus\{a\} , \s \forall y\in \Sp^{n-1} . 
\end{equation}

Let us set $u_{k,\ep}:=u_k * \rho_\ep$. We claim that, if we set
$$\ep_k:=(n-p_k)^{\frac{3}{\delta(n-p_k)}},$$
then it is possible to choose $a_k\in B^n _\delta$ in such a way that the maps $\tilde{u}_k(x):=\pi_{a_k} (u_{k,\ep_k}(x))$ satisfy all the requirements.

In what follows, we always take $k$ sufficiently large such that $\ep_k <{\rm dist}(\Omega, \R^{n+m}\setminus U) $, so that $u_{k, \ep_k}:\Omega \to \R^n$ and $\tilde{u}_k:\Omega\to \Sp^{n-1}$ are well-defined in $\Omega$. 

We start by observing that if $a_k$ is a regular value of the smooth map $u_{k,\ep_k}$, then the map $\tilde{u}_k$ turns out to be smooth outside $\Sigma_k:= u_{k,\ep_k}^{-1}(a_k)$, which is a smooth orientable $m$-submanifold of $\Omega$. Moreover, since in this case $u_{k,\ep_k}$ looks like a projection onto $(T_x \Sigma_k)^\perp$ locally around every $x\in \Sigma_k$ (namely, it coincide with this projection up to a diffeomorphism), we deduce that $\abs{\deg (\tilde{u}_k,S)}=1$ for every embedded $(n-1)$-sphere linked with a single connected component of $\Sigma_k$ (for example this holds when $S$ is a connected component of $u_{k,\ep_k}^{-1}(\partial B^n _{r}(a_k))$ for some small radius $r>0$). Thus, from Proposition~\ref{prop:Jacobian_integral} we deduce that $\star J \tilde{u}_k = \gamma_n \lcurr \Sigma_k \rcurr$, for a suitable choice of the orientation on each connected component of $\Sigma_k$.

Now we need to show that it is possible to choose these regular values $a_k$ so that also \eqref{est:approx1}, \eqref{est:approx2} and \eqref{est:approx3} hold.

First of all, we prove that \eqref{est:approx1} is true independently of the choice of $a_k$. To this end, let us consider the numbers $r_1<n<r_2$ defined by
$$\frac{1}{r_1}:=\frac{1}{n}+\frac{1}{3(n+m)},
\qquad\text{and}\qquad
\frac{1}{r_2}:=\frac{1}{n}-\frac{1}{3(n+m)},$$
so that the Sobolev embedding \cite[Theorem~8.24]{Leoni-book} yields
$$[u]_{W^{1/3,r_2}(U)}\leq C \|\nabla u\|_{L^{r_1}(U)} \qquad \forall u\in W^{1,r_1}(U;\R^n),$$
for some constant $C>0$ depending only on $U,n,m$. Moreover, Lemma~\ref{lemma:sobolev-convolution} yields
$$\|u_k-u_{k,\ep}\|_{L^{r_2}(\Omega)}\leq \|\rho\|_{L^\infty} \cdot \ep^{1/3}[u_k]_{W^{1/3,r_2}(U)}.$$

Therefore, if $k$ is sufficiently large to ensure that $r_1<p_k<n<q_k<r_2$, from the last two estimates and Jensen's inequality we deduce that
$$\|u_k-u_{k,\ep_k}\|_{L^{q_k}(\Omega)}\leq C {\ep_k}^{1/3}\|\nabla u_k\|_{L^{p_k}(U)},$$
where $C>0$ is a constant depending only on $\Omega$, $U$, and the space dimensions $n$ and $m$.

As a consequence, the second inequality in \eqref{est:pi_a} and the definition of $\ep_k$ imply that
$$\|u_k-\pi_a(u_{k,\ep_k})\|_{L^{q_k}(\Omega)} \leq C (n-p_k)^{\frac{1}{\delta(n-p_k)}}\|\nabla u_k\|_{L^{p_k}(U)}\leq C (n-p_k)^{\frac{1}{\delta(n-p_k)}-\frac{1}{p_k}}E^{\frac{1}{p_k}},$$
for every $a\in B_\delta ^n$, where $C$ is a constant depending only on $\Omega,U,n,m$. In particular, \eqref{est:approx1} holds true for sufficiently large $k$, independently of the choice of $a_k$.

Now we focus on the estimates \eqref{est:approx2} and \eqref{est:approx3}, and we show that both are true for a set of $a_k$ with sufficiently large measure, so that in the end we have a set of positive measure of $a_k$ for which all the statements are true.

Concerning \eqref{est:approx2}, we exploit the coarea formula, which yields (see \cite[Section~7.5]{ABO2-singularities} and \cite[Section~2.7]{2005-Indiana-ABO})
%
$$\int_{B^n _\delta} \mathcal{H}^{m}(u_{k,\ep}^{-1}(a))\,da\leq \int_{B^n _1} \mathcal{H}^{m}(u_{k,\ep}^{-1}(a))\,da =\M(\star Ju_{k,\ep})\leq \int_{\Omega} |\nabla u_{k,\ep}(x)|^n\,dx.$$

Moreover, from basic properties of convolution and the fact that $\|u_k\|_{L^\infty(U)}=1$, we obtain that
\begin{multline*}
\int_{\Omega} |\nabla u_{k,\ep}(x)|^n\,dx =\int_{\Omega} |u_k* \nabla \rho_\ep (x)|^{n-p_k}\cdot |\nabla u_{k} * \rho_\ep (x)|^{p_k}\,dx \\
\leq \|\nabla \rho_\ep\|_{L^1} ^{n-p_k} \int_{U} |\nabla u_{k,\ep}(x)|^{p_k}\,dx\leq \frac{\|\nabla \rho\|_{L^1} ^{n-p_k}}{\ep^{n-p_k}} \int_{U} |\nabla u_{k}(x)|^{p_k}\,dx,
\end{multline*}
and hence
\begin{equation}\label{est:a-1}
\int_{B^n _\delta} \mathcal{H}^{m}(u_{k,\ep_k}^{-1}(a))\,da\leq \frac{\|\nabla \rho\|_{L^1} ^{n-p_k}}{\ep_k ^{n-p_k}} \cdot \frac{E}{n-p_k} = \frac{\|\nabla \rho\|_{L^1} ^{n-p_k} \cdot E}{(n-p_k)^{1+3/\delta}}=\frac{\|\nabla \rho\|_{L^1} ^{n-p_k} \cdot E}{[(1-\delta)(n-\tilde{p}_k)]^{1+3/\delta}},
\end{equation}
which is an averaged version (with respect to $a$) of \eqref{est:approx2}.

Finally, we need to estimate the $W^{1,\tilde{p}_k}$ energy of $\tilde{u}_k$, which is the most delicate step, and is the reason for which we need to reduce the integrability exponent from $p_k$ to $\tilde{p}_k$.

To this end, we fix a non-decreasing smooth function $\psi_\delta:[0,+\infty)\to [0,1]$ such that
\begin{equation}\label{defn:psi_delta}
\psi_\delta(t):=\begin{cases}
0 &\text{if } t\in [0,\delta],\\
1 &\text{if } t\in [1-2\delta,+\infty),
\end{cases}
\qquad\text{and}\qquad
\psi_\delta '(t)\leq \frac{1}{1-4\delta} \qquad \forall t\in [0,+\infty).
\end{equation}
Then we consider the maps $\Psi_{\delta,a}:B_1 ^n\to B_1 ^n$ defined by
$$\Psi_{\delta,a}(x):=a+\psi_\delta(|x-a|)t_a(x)\frac{x-a}{\abs{x-a}}=a+\psi_\delta(|x-a|)(\pi_a(x)-a),$$
and we observe that
\begin{align*}
\nabla \Psi_{\delta,a}(x) &= \psi_\delta '(|x-a|) t_a(x)\frac{x-a}{|x-a|}\otimes \frac{x-a}{|x-a|} + \frac{\psi_\delta(|x-a|)}{|x-a|}t_a(x) \biggl(\textbf{I}_{n\times n }-\frac{x-a}{|x-a|}\otimes \frac{x-a}{|x-a|}\biggr)\\
&\quad \mbox{}+\psi_\delta(|x-a|)\nabla t_a(x) \otimes \frac{x-a}{\abs{x-a}},
\end{align*}
for every $x\in B_1 ^n$ and every $a\in B_\delta ^n$. Therefore, from \eqref{est:t_a} and \eqref{defn:psi_delta} we deduce that
\begin{align*}
\sup_{|\xi|=1} |\nabla \Psi_{\delta,a}(x) \xi|&\leq \max\biggl\{ \psi_\delta '(|x-a|),\frac{\psi_\delta(|x-a|)}{|x-a|}\biggr\} |t_a(x)| + \psi_\delta(|x-a|)|\nabla t_a(x)|\\
&\leq \frac{1+|a|}{1-4\delta} + \frac{C|a|}{1-4\delta}\leq \frac{1+C\delta}{1-4\delta} ,
\end{align*}
for every $x \in B^n _1$. Now let us set
$$g_{k,\ep,a}(x):=\bigl(1-\psi_\delta(|u_{k,\ep}(x)-a|)\bigr) \bigl(\pi_{a}(u_{k,\ep}(x))-a\bigr),$$
so that we have
\begin{equation}\label{eq:pi_a-split}
\pi_a(u_{k,\ep})=\Psi_{\delta,a}(u_{k,\ep})+g_{k,\ep,a},
\end{equation}
and we can estimate separately the two addenda.

For the first one, we exploit the previous estimate and Hölder inequality, so we obtain that
\begin{align*}
\int_{\Omega_1} |\nabla (\Psi_{\delta,a}(u_{k,\ep}))|^{\tilde{p}_k} &= \int_{\Omega_1} |\nabla \Psi_{\delta,a}(u_{k,\ep}) \nabla u_{k,\ep}|^{\tilde{p}_k} \leq \biggl(\frac{1+C\delta}{1-4\delta}\biggr)^{\tilde{p}_k} \int_{\Omega_1} |\nabla u_{k,\ep}|^{\tilde{p}_k}\\
&\leq |\Omega_1|^{1-\frac{\tilde{p}_k}{p_k}}\biggl(\frac{1+C\delta}{1-4\delta}\biggr)^{\tilde{p}_k} \biggl(\int_{\Omega_1} |\nabla u_{k,\ep}|^{p_k}\biggr)^{\frac{\tilde{p}_k}{p_k}}\\
&\leq |\Omega|^{1-\frac{\tilde{p}_k}{p_k}} \biggl(\frac{1+C\delta}{1-4\delta}\biggr)^{n} \biggl(\int_{\Omega_2} |\nabla u_{k}|^{p_k}\biggr)^{\frac{\tilde{p}_k}{p_k}},
\end{align*}
provided $\ep<\dist(\Omega_1,\partial \Omega_2)$. Thus, multiplying both sides by $(n-\tilde{p}_k)$, we find that
\begin{multline*}
(n-\tilde{p}_k)\int_{\Omega_1} |\nabla (\Psi_{\delta,a}(u_{k,\ep}))|^{\tilde{p}_k}\leq |\Omega|^{1-\frac{\tilde{p}_k}{p_k}} \biggl(\frac{1+C\delta}{1-4\delta}\biggr)^{n} \cdot \frac{n-p_k}{1-\delta} \cdot \biggl(\int_{\Omega_2} |\nabla u_{k}|^{p_k}\biggr)^{\frac{\tilde{p}_k}{p_k}}\\
\leq \bigl((n-p_k)|\Omega| \bigr)^{1-\frac{\tilde{p}_k}{p_k}} \biggl(\frac{1+C\delta}{1-4\delta}\biggr)^{n} \cdot \frac{1}{1-\delta} \cdot \biggl((n-p_k)\int_{\Omega_2} |\nabla u_{k}|^{p_k}\biggr)^{\frac{\tilde{p}_k}{p_k}}.
\end{multline*}

Now we set
$$L_\delta :=\biggl(\frac{1+C\delta}{1-4\delta}\biggr)^{n} \cdot \frac{1}{1-\delta},$$
so in particular $L_\delta \to 1^+$ as $\delta\to 0^+$, and the previous estimate implies that
\begin{equation}\label{est:Psi(u)}
\limsup_{k\to +\infty}\, (n-\tilde{p}_k)\int_{\Omega_1} |\nabla (\Psi_{\delta,a}(u_{k,\ep}))|^{\tilde{p}_k} - L_\delta (n-p_k)\int_{\Omega_2} |\nabla u_{k}|^{p_k} \leq 0,
\end{equation}
uniformly with respect to $a \in B^n _\delta$ and $\ep \in (0,\dist(\Omega_1,\partial \Omega_2))$.

As for $g_{k,\ep,a}$, recalling \eqref{est:pi_a} and \eqref{defn:psi_delta}, we obtain that
\begin{align*}
|\nabla g_{k,\ep,a}(x)| &\leq (1+|a|) |\psi_\delta '(|u_{k,\ep}(x)-a|)| \cdot |\nabla u_{k,\ep}(x)| + |\nabla \pi_a(u_{k,\ep}(x))|\cdot |\nabla u_{k,\ep}(x)|\\
&\leq \biggl( \frac{1+\delta}{1-4\delta} + \frac{C}{|u_{k,\ep}(x)-a|}\biggr)|\nabla u_{k,\ep}(x)| \leq \frac{C}{|u_{k,\ep}(x)-a|}|\nabla u_{k,\ep}(x)|,
\end{align*}
for some universal constant $C>0$.

Moreover, we observe that $\psi_\delta(|u_{k,\ep}-a|)=1$ where $|u_{k,\ep}|\geq 1-\delta$, because in this case we have that $|u_{k,\ep}-a|\geq|u_{k,\ep}|-|a|\geq 1-2\delta$. As a consequence, we deduce that $g_{k,\ep,a}$ is supported where $|u_{k,\ep}|<1-\delta$, so we have that
$$\int_{\Omega} |\nabla g_{k,\ep,a}(x)|^{\tilde{p}_k}\,dx\leq C \int_{\{|u_{k,\ep}|<1-\delta\}}\frac{|\nabla u_{k,\ep}(x)|^{\tilde{p}_k}}{|u_{k,\ep}(x)-a|^{\tilde{p}_k}}\,dx.$$

Now we observe that, for every $z\in B^n _1$, it holds that 
$$\int_{B^n _\delta} \frac{1}{|z-a|^{\tilde{p}_k}}\,da \leq \int_{B^n _\delta} \frac{1}{|a|^{\tilde{p}_k}}\,da = \omega_{n-1} \frac{\delta^{n-\tilde{p}_k}}{n-\tilde{p}_k},$$
and hence
$$\int_{B^n _\delta} da \int_{\Omega} |\nabla g_{k,\ep,a}(x)|^{\tilde{p}_k}\,dx \leq C \omega_{n-1} \frac{\delta^{n-\tilde{p}_k}}{n-\tilde{p}_k}  \int_{\{|u_{k,\ep}|<1-\delta\}} |\nabla u_{k,\ep}(x)|^{\tilde{p}_k}\,dx.$$

Moreover, from Lemma~\ref{lemma:sobolev-convolution} we deduce that
$$|\{|u_{k,\ep}|<1-\delta\}|\leq \frac{1}{\delta^{p_k}}\int_{\Omega} |u_{k}(x)-u_{k,\ep}(x)|^{p_k}\,dx \leq \frac{\ep^{p_k}}{\delta^{p_k}} \int_{U}|\nabla u_k(x)|^{p_k}\,dx,$$
so Hölder inequality yields
\begin{align*}
&\int_{B^n _\delta} da \int_{\Omega} |\nabla g_{k,\ep,a}(x)|^{\tilde{p}_k}\,dx \\
&\hspace{7em}\leq C \omega_{n-1}\cdot \frac{\delta^{n-\tilde{p}_k}}{n-\tilde{p}_k} \cdot |\{|u_{k,\ep}|<1-\delta\}|^{1-\frac{\tilde{p}_k}{p_k}} \biggl(\int_{\{|u_{k,\ep}|<1-\delta\}} |\nabla u_{k,\ep}|^{p_k}\,dx\biggr)^{\frac{\tilde{p}_k}{p_k}}\\
&\hspace{7em}\leq C \omega_{n-1}\cdot \frac{\delta^{n-\tilde{p}_k}}{n-\tilde{p}_k} \cdot \frac{\ep^{p_k-\tilde{p}_k}}{\delta^{p_k-\tilde{p}_k}} \biggl(\int_{U}|\nabla u_k|^{p_k}\,dx\biggr)^{\frac{\tilde{p}_k}{p_k}}.
\end{align*}

In particular, when $\ep=\ep_k$, we obtain that
$$\int_{B^n _\delta} da \int_{\Omega} |\nabla g_{k,\ep_k,a}(x)|^{\tilde{p}_k}\,dx \leq C \omega_{n-1} \frac{\delta^{n-p_k}}{1+\delta} (n-p_k)^2  \biggl(\int_{U}|\nabla u_k|^{p_k}\,dx\biggr)^{\frac{\tilde{p}_k}{p_k}}.$$

From the last estimate and \eqref{est:a-1}, we deduce that for every $k\in\N$ there exists $a_k\in B_\delta ^n$ (and actually a set of positive measure of such points) that is a regular value for $u_{k,\ep_k}$ and such that
$$\mathcal{H}^{m}(\Sigma_k)=\mathcal{H}^{m}(u_{k,\ep_k}^{-1}(a_k)) \leq \frac{3}{|B^n _1|\delta^n} \cdot \frac{\|\nabla \rho\|_{L^1} ^{n-p_k}\cdot E}{[(1-\delta)(n-p_k)]^{1+3/\delta}},$$
that is \eqref{est:approx2}, and
\begin{align*}
\int_{\Omega} |\nabla g_{k,\ep_k,a_k}(x)|^{\tilde{p}_k}\,dx &\leq  \frac{3}{|B^n _1|\delta^n} \cdot C \omega_{n-1} \frac{\delta^{n-p_k}}{1+\delta} (n-p_k)^2 \biggl(\int_{U}|\nabla u_k|^{p_k}\,dx \biggr)^{\frac{\tilde{p}_k}{p_k}}\\
&\leq  \frac{C}{\delta^{p_k}(1+\delta)} (n-p_k) E^{\tilde{p}_k/p_k},
\end{align*}
for some constant $C>0$, independent of $k$. In particular, it holds that
$$\lim_{k\to +\infty} \int_{\Omega} |\nabla g_{k,\ep_k,a_k}(x)|^{\tilde{p}_k}\,dx=0,$$
and this, together with \eqref{eq:pi_a-split} and \eqref{est:Psi(u)}, yields \eqref{est:approx3}.
\end{proof}

\begin{rem}\label{rem:flat_approx}
By combining \eqref{est:approx1}, \eqref{est:approx3} and \eqref{th:BN_invent}, applied with $p=\tilde{p}_k$ and $q=q_k$, we also deduce that
\begin{equation}\label{approx_flat}
\lim_{k\to +\infty} \F_{\Omega}(\star Ju_k- \star J\tilde{u}_k)=0.
\end{equation}
\end{rem}

\section{The zero-dimensional case \texorpdfstring{$(m=0)$}{m=0} } \label{sec: zero dim case}

In this section we consider the case in which $m=0$, so we deal with maps defined on a domain of $\R^n$ with values into $\Sp^{n-1}$, for which the Jacobian is a $0$-dimensional current.

In this setting we prove Theorem~\ref{teo:X_M+X_F}, which is a key result for this paper, and provides in some sense a $W^{1,p}$ counterpart of the estimates in \cite{Sandier-JFA-1998,1999-SIAM-Jerrard}, playing the role of \cite[Lemma~3.10]{2005-Indiana-ABO} and \cite[Proposition~3.2]{Jerrard-Soner-calcvar} in the Ginzburg-Landau framework.

Theorem~\ref{teo:X_M+X_F}, together with the approximation result of Proposition~\ref{prop:approx}, immediately provides a proof of statement~$(i)$ in Theorem~\ref{main_res} in the case $m=0$, but is also the main ingredient for the proof in the case $m>0$, which then follows just by pursuing the strategy of \cite{2005-Indiana-ABO}, with minor adjustments.

Before stating the result, we state two preliminary lemmas which are needed for the proof. The first one is a simple estimate, similar to  \cite[Lemma~1.1]{HardtLinChangyou}, which relates the $p$-energy of a map on the boundary of a set to the absolute value of the Jacobian computed on that set.

\begin{lemma}\label{lemma:W1p_boundary}
Let $\Omega \subset \R^n$ be an open set, let $p\in [n-1,n)$, and let $u\in W^{1,p}(\Omega;\Sp^{n-1})$ be a map that is smooth outside a finite set. Let also $A\Subset \Omega$ be an open set whose boundary $\Gamma:=\partial A\subset \Omega$ is Lipschitz and does not intersect the singular set of $u$. Then it holds that
\begin{equation}\label{est:boundary-energy-Jac}
\int_{\Gamma} |\nabla u(x)|^p \,d\mathcal{H}^{n-1}(x) \geq n^{\frac{p}{n-1}} (n-1)^{\frac{p}{2}}\frac{\abs{\star Ju(A)}^{\frac{p}{n-1}}}{\mathcal{H}^{n-1}(\Gamma)^{\frac{p}{n-1}-1}}. 
\end{equation}
\end{lemma}

\begin{proof}
We first observe that, by Proposition~\ref{prop:Jacobian_integral}, $\star Ju$ is a finite measure supported on the finite singular set of $u$, so in particular $\star Ju(A)$ is well-defined, and is equal to the limit of $\langle \star Ju ,\psi\rangle$ as $\psi$ approximates the characteristic function of $A$. Since $ju$ is a smooth $(n-1)$-form around $\Gamma$, from the definition of $\star Ju$ we deduce that
$$\abs{\star Ju(A)}=\left|\frac{1}{n} \int_{\Gamma} ju \right|\leq \frac{1}{n} \int_{\Gamma} |ju| .  $$
Now we exploit \cite[Lemma~4.3]{Bre-Mi}, which states that
$$|ju|\leq \frac{1}{(n-1)^{\frac{n-1}{2}}} |\nabla u| ^{n-1},$$
so we obtain that
$$\abs{\star Ju(A)}\leq \frac{1}{n(n-1)^{\frac{n-1}{2}}}  \int_{\Gamma} |\nabla u(x)| ^{n-1} \,d\mathcal{H}^{n-1}(x).$$

Finally, H\"older inequality yields
$$\abs{\star Ju(A)}\leq \frac{1}{n(n-1)^{\frac{n-1}{2}}} 
 \mathcal{H}^{n-1}(\Gamma)^{1-\frac{n-1}{p}}\left(\int_{\Gamma} |\nabla u(x)| ^{p} \,d\mathcal{H}^{n-1}(x)\right)^{\frac{n-1}{p}}, $$
which is equivalent to \eqref{est:boundary-energy-Jac}.
\end{proof}

The second lemma we need is just an estimate from below for a sum emerging from the computations in the proof of Theorem~\ref{teo:X_M+X_F}. We prefer to state it separately and prove it in the Appendix because, even though it is elementary, it requires a bit of work, which, however, is not directly related to the arguments in the proof of Theorem~\ref{teo:X_M+X_F}.

\begin{lemma}\label{lemma:a_i}
Let $K\in\N$ be a nonnegative integer and let $(a_0,\dots,a_k)\in\N^{K+1}$ be nonnegative integers. Let us set $S_k:=a_0+\dots+a_k$ for every $k\in\{0,\dots,K\}$ and $a_{K+1}:=0$.

Then for every $\beta\in (1,2)$ and every $\lambda \in (3/4,1)$ it turns out that
\begin{equation}\label{th:lemma-a_i}
\sum_{k=0} ^{K} \frac{\max\{S_k-a_{k+1},0\}^{\beta}}{S_k ^{\beta-1}}\lambda^k \geq \frac{2\lambda-1}{2\lambda} \sum_{k=0} ^{K} a_{k} \lambda^{k},
\end{equation}
where the fraction in the left-hand side is intended to vanish if $S_k=0$.
\end{lemma}

We are now ready to state and prove the main result of this section.

\begin{teo}\label{teo:X_M+X_F}
Let $n\geq 2$ be an integer, and let $\alpha\in (\sqrt{3}/2,1)$ and $p\in (n-1,n)$ be such that
\begin{equation}\label{prop:alpha_p0}
n-p\leq\frac{1}{4} \quad\text{and}\quad \frac{p(n-p)}{n-1} \frac{\alpha^{-\frac{1}{n-p}}}{\alpha^{-\frac{1}{n-p}+2}-1}\leq \frac{(1-\alpha)(2\alpha^2-1)}{4\alpha}. 
\end{equation}

Then there exists a constant $C=C(n,\alpha)>0$, depending only on $n$ and $\alpha$, such that the following statement holds.

For every bounded open set $\Omega\subseteq\R^n$ with Lipschitz boundary, and every map $u\in W^{1,p}(\Omega ; \Sp^{n-1})$ that is smooth outside a finite set, there exist an integral $0$-current $X$ and an integral $1$-current $S$, which is a finite sum of segments, such that $\star Ju/\gamma_n=X+\partial S$ and the following estimates hold
\begin{equation}\label{th:mass}
\M_{\Omega}(X)\leq C (n-p)\int_{\Omega} |\nabla u(x)|^p\,dx,
\end{equation}
and
\begin{equation}\label{th:flat}
\M_{\Omega}(S)\leq C \alpha^{\frac{1}{2(n-p)}} (n-p) \int_{\Omega} |\nabla u(x)|^p\,dx.
\end{equation}

In addition, it also holds that
\begin{multline}\label{th:mass_sharp}
(n-1)^{\frac{p}{2}} \omega_{n-1} \biggl(\alpha \M_{\Omega}(X)-\frac{p(2^{n-p}-1)}{n-1}\biggr)\\
\leq \Biggl(1+C\frac{2^{n-p}-1}{\sqrt{n-p}} \biggl(1+ 2\frac{\log\bigl(\M_\Omega(\star Ju /\gamma_n)\bigr)}{\abs{\log(n-p)}}\biggr) \Biggr) (n-p) \int_{\Omega} |\nabla u(x)|^p\,dx.
\end{multline}
\end{teo}

\begin{proof}
We observe that, from our assumptions and Proposition~\ref{prop:Jacobian_integral}, it follows that $\star Ju/\gamma_n$ is an integral $0$-current, namely a finite sum of Dirac deltas.



Let us set
\begin{equation}\label{defn:alpha_k}
\alpha_k:=\alpha^{\frac{k}{n-p}} \qquad \forall k\in\N.
\end{equation}


Given an integral $0$-current $T$, let us set
$$T^+:=\{y\in \Omega : T(\{y\})>0\}
\qquad \text{and} \qquad
T^- :=\{z\in \Omega:T(\{z\})<0\},$$
and let us consider the minimization problem
$$m(T):=\min\{|y-z|:(y,z)\in T ^+ \times T ^- \cup T ^+ \times \partial \Omega \cup \partial\Omega \times T ^- \}.$$

We observe that a minimizer for $m(T)$ always exists, because both $T^+$ and $T^-$ are finite sets, and $\partial \Omega$ is compact.

In order to construct the currents $X$ and $S$, we argue iteratively in the following way.

First, we set $S_1=0$ and $X_1=\star Ju /\gamma_n$. Then, assuming that we have already defined $S_i$ and $X_i$ for some $i\in\N$, if either $X_i=0$ or $m(X_i)>\alpha_1=\alpha^{1/(n-p)}$, we set $S:=S_i$ and $X:=X_i$. Otherwise, namely if $m(X_i)\leq\alpha_1$, we take a minimizer $(y_i,z_i)$ for $m(X_i)$ and we set
$$S_{i+1}:=S_i + [y_i,z_i]\qquad\text{and}\qquad X_{i+1}:=X_i-\partial [y_i,z_i]=\star Ju/\gamma_n -\partial S_{i+1},$$
where $[y_i,z_i]$ denotes the segment joining $y_i$ and $z_i$, oriented in such a way that $\partial [y_i,z_i] = \delta_{y_i}-\delta_{z_i}$ in $\overline{\Omega}$.


We remark that the segment $[y_i,z_i]$ is contained in $\Omega$, with the possible exception of at most one endpoint, because we know that at least one of the endpoints belongs to $\supp X_i \subset \Omega$, and if a point in the interior of the segment $[y_i,z_i]$ belonged to $\partial \Omega$, this would contradict the minimality of $(y_i,z_i)$.

We also point out that at each iteration the total mass of $X_{i+1}$ decreases by either one (in the case in which $y_i\in \partial \Omega$ or $z_i\in \partial \Omega$) or two (in the case in which $(y_i,z_i)\in X_i ^+\times X_i ^-$). In any case, we have that $\M(X_{i+1})\leq \M(X_{i})-1$, so this procedure certainly ends after a finite number of iterations, providing us the currents $X$ and $S$, that can be written as
$$S=\sum_{i=1} ^{N} [y_i,z_i],$$
for some $N\in\N$. As a consequence, $\partial S$ can be written as
$$\partial S=\sum_{i=1} ^{N} \mathbbm{1}_{\Omega}(y_i)\delta_{y_i} - \mathbbm{1}_{\Omega}(z_i) \delta_{z_i},$$
where $\mathbbm{1}_\Omega(x)=1$ if $x \in \Omega$ and $\mathbbm{1}_\Omega(x)=0$ if $x \in \partial\Omega$. Moreover, we have that $|y_i -z_i| \leq \alpha_1$ for every $i\in \{1,\dots,N\}$, while
\begin{equation}\label{eq:m(X_M)}
m(X)>\alpha_1 \quad \text{or}\quad X=0,
\end{equation}
because this is the condition ending our iterative construction.

Now, for every $k\in\N^+$ we set
\begin{equation*}
I_k:=\{i\in \{1,\dots,N\}:y_i\in \Omega, |y_i-z_i|\in (\alpha_{k+1} ,\alpha_{k}]\},\qquad E_k:=\bigcup_{i\in I_k} \{y_i\},\qquad e_k:=\mathcal{H}^0(I_k).
\end{equation*}

In words, the set $I_k$ contains all the indices $i$ corresponding to positive atoms $y_i$ of $\partial S$ which are paired with some $z_i$ at distance between $\alpha_{k+1}$ and $\alpha_k$, the set $E_k$ is the (finite) subset of $\Omega$ supporting these positive atoms, and $e_k$ is the cardinality of $I_k$, namely their total mass.

We remark that a single point $y\in \Omega$ may appear more than once in the sequence $y_1,\dots,y_m$ (if $\partial S(\{y\})>1$) and each time it may be paired with a different point $z_{i}$. For this reason, the sets $E_k$ are not necessarily pairwise disjoint, and their cardinality may be less than $e_k$.

We also define $E_0:=X ^+$ and $e_0:=X(E_0)$, and we set
$$K:=\max\{k\in\N:e_k>0\},\qquad S_k:= \sum_{j=0} ^{k} e_j.$$

Finally, for every $k\in\{0,\dots,K\}$ and every $t\in (0,\alpha_{k+1})$ we set
$$D_{k}(t):=\{x\in \Omega:\dist(x, E_0\cup\dots\cup E_k)< t\}.$$

We point out that $\dist(y_i,\partial\Omega)\geq |y_i-z_i|> \alpha_{k+1}$ for every $i\in I_k$, because of the minimality property of $(y_i,z_i)$, and the fact that $y_i\notin \partial\Omega$. As a consequence, for every $t\in (0,\alpha_{k+1})$ the set $D_k(t)$ is the union of a finite number of balls that are contained in $\Omega$, and we have that
$$\partial D_{k}(t) =\{x\in \Omega:\dist(x, E_0\cup\dots\cup E_k)= t\}$$

Moreover, the function $x\mapsto \dist(x, E_0\cup\dots\cup E_k)$ is $1$-Lipschitz continuous in $\Omega$, and its gradient has unit norm almost everywhere.

Therefore, for every $k\in \{0,\dots,K\}$, the coarea formula yields
\begin{equation}\label{est:energy-coarea}
\int_{\Omega} |\nabla u(x)|^p\,dx  \geq \int_{0} ^{\alpha_{k+1}} dt \int_{\partial D_k(t)} |\nabla u(x)|^p \,d\mathcal{H}^{n-1}(x)\geq \int_{\alpha_{k+2}} ^{\alpha_{k+1}} dt \int_{\partial D_k(t)} |\nabla u(x)|^p \,d\mathcal{H}^{n-1}(x).
\end{equation}

Since $u$ is smooth around $\partial D_k(t)$ for all but finitely many $t>0$, from Lemma~\ref{lemma:W1p_boundary} we get that
\begin{equation}\label{est:energy-1}
\int_{\Omega} |\nabla u(x)|^p\,dx  \geq n^{\frac{p}{n-1}} (n-1)^{\frac{p}{2}} \int_{\alpha_{k+2}} ^{\alpha_{k+1}} \frac{\abs{\star Ju (D_k(t))}^{\frac{p}{n-1}}}{\mathcal{H}^{n-1}(\partial D_k(t))^{\frac{p}{n-1}-1}} dt.
\end{equation}

Now we write the Jacobian as
$$\star Ju/\gamma_n = \partial S +X = \partial S \mres {(\partial S) ^+} + \partial S \mres {(\partial S) ^-} + X \mres {X ^+} + X \mres {X ^-},$$
and we observe that the construction of $S$ and $X$ ensures that $(\partial S) ^+ \cup X ^+ = (\star Ju)^+$ and that $(\partial S )^- \cup X ^- = (\star Ju)^-$, while of course $(\star Ju)^+\cap (\star Ju)^-= \varnothing$.

Moreover, we know that $D_k(\alpha_{k+1})\cap X ^-=\varnothing$. Indeed, let us assume by contradiction that there exist $z\in X ^-$ and $y\in E_0\cup \dots \cup E_k$ such that $|y-z|\leq \alpha_{k+1}$. Then either $y\in E_0$ or $y=y_i$ for some $i\in I_1\cup\dots \cup I_k$. In the first case, we find a contradiction with \eqref{eq:m(X_M)}. In the second case we deduce that $|y_i-z|\leq \alpha_{k+1}< |y_i-z_i|$, which contradicts the minimality property of $(y_i,z_i)$.

With a similar contradiction argument, we also deduce that $z_i\notin D_k(\alpha_{k+1})$ for every $i\in I_{1}\cup \dots \cup I_k$.

Therefore, for every $k\in \{1,\dots,K\}$ and every $t\in (0,\alpha_{k+1})$ we have that
\begin{align}
\star Ju(D_k(t))/\gamma_n &= X\mres {X ^+}(D_k(t))+ \partial S (D_k(t))\nonumber\\
&=e_0 + \sum_{j=1} ^{k} \sum_{i \in I_{j}} \mathbbm{1}_{D_k(t)}(y_i)-\mathbbm{1}_{D_k(t)}(z_i) + \sum_{j=k+1} ^{K} \sum_{i \in I_{j}} \mathbbm{1}_{D_k(t)}(y_i)-\mathbbm{1}_{D_k(t)}(z_i)\quad\nonumber\\
&=e_0 + \sum_{j=1} ^{k} e_j + \sum_{j=k+1} ^{K} \sum_{i \in I_{j}} \mathbbm{1}_{D_k(t)}(y_i)-\mathbbm{1}_{D_k(t)}(z_i).\label{est:Ju-1}
\end{align}

Now we observe that $\mathbbm{1}_{D_k(t)}(y_i)-\mathbbm{1}_{D_k(t)}(z_i)=0$ unless $t$ is between $\dist(y_i, E_0\cup\dots\cup E_k)$ and $\dist(z_i, E_0\cup\dots\cup E_k)$. As a consequence, we deduce that the set
$$T_{k,i}:=\{t>0: \mathbbm{1}_{D_k(t)}(y_i)-\mathbbm{1}_{D_k(t)}(z_i)< 0\}$$
satisfies
$$\Leb^1(T_{k,i})\leq \abs{\dist(y_i, E_0\cup\dots\cup E_k)-\dist(z_i, E_0\cup\dots\cup E_k)} \leq |y_i-z_i|,$$
where $\Leb^1$ is the Lebesgue measure, and hence
\begin{equation}\label{est:Leb-T}
\Leb^1(T_{k,i})\leq \alpha_{j}\qquad \forall i\in I_j \quad \forall j\in\{k+1,\dots,K\}.
\end{equation}

Continuing from \eqref{est:Ju-1} we get that
\begin{equation}\label{est:Ju-1.5}
\star Ju(D_k(t))/\gamma_n \geq S_k - \sum_{j=k+1} ^{K} \sum_{i \in I_{j}} \mathbbm{1}_{T_{k,i}}(t)
\geq S_k - e_{k+1} -\sum_{j=k+2} ^{K} \sum_{i \in I_{j}} \mathbbm{1}_{T_{k,i}}(t),
\end{equation}
which implies that
$$\abs{\star Ju(D_k(t))}\geq \gamma_n \max\{\star Ju(D_k(t)),0\}
\geq \gamma_n\max\Biggl\{S_k - e_{k+1} -\sum_{j=k+2} ^{K} \sum_{i \in I_{j}} \mathbbm{1}_{T_{k,i}}(t),0 \Biggr\},$$
and hence, since $\max\{a-b,0\}^q \geq \max\{a,0\}^q- q\max\{a,0\}^{q-1}b$ for every $q\geq 1$ and every $a\in \R$ and $b\geq 0$, we obtain that
\begin{align}
\abs{\star Ju(D_k(t))}^{\frac{p}{n-1}}&\geq \gamma_n^\frac{p}{n-1} \Biggl(\max\bigl\{S_k - e_{k+1} ,0\bigr\}^{\frac{p}{n-1}}\nonumber\\
&\quad\mbox{}- \frac{p}{n-1}\max\bigl\{S_k - e_{k+1} ,0\bigr\}^{\frac{p}{n-1}-1} \sum_{j=k+2} ^{K} \sum_{i \in I_{j}} \mathbbm{1}_{T_{k,i}}(t)\Biggr)\nonumber\\
&\geq \gamma_n^{\frac{p}{n-1}}\Biggl(\max\bigl\{S_k - e_{k+1} ,0\bigr\}^{\frac{p}{n-1}} - \frac{p}{n-1} S_k^{\frac{p}{n-1}-1} \sum_{j=k+2} ^{K} \sum_{i \in I_{j}} \mathbbm{1}_{T_{k,i}}(t)\Biggr).\quad \label{est:Ju-2}
\end{align}

On the other hand, since $D_k(t)$ is the union of at most $S_k$ balls of radius $t$, it holds that
\begin{equation}\label{est:meas_spheres}
\mathcal{H}^{n-1}(\partial D_k(t))\leq S_k \omega_{n-1} t^{n-1}
\end{equation}

Plugging \eqref{est:Ju-2} and \eqref{est:meas_spheres} into \eqref{est:energy-1} we obtain that
\begin{equation}\label{est:energy-2}
\int_{\Omega} |\nabla u(x)|^p\,dx \geq (n-1)^{\frac{p}{2}} \omega_{n-1} \int_{\alpha_{k+2}} ^{\alpha_{k+1}} \left[ \frac{\max\bigl\{S_k - e_{k+1} ,0\bigr\}^{\frac{p}{n-1}}}{S_k ^{\frac{p}{n-1}-1} t^{p-n+1}}  - \frac{p}{n-1} \sum_{j=k+2} ^{K} \sum_{i \in I_{j}} \frac{\mathbbm{1}_{T_{k,i}}(t)}{t^{p-n+1}}\right]dt.
\end{equation}

From \eqref{est:Leb-T} we deduce that for every $k\in\{0,\dots,K-2\}$, for every $j\in \{k+2,\dots,K\}$ and for every $i\in I_{j}$ it holds that
$$\int_{\alpha_{k+2}} ^{\alpha_{k+1}} \frac{\mathbbm{1}_{T_{k,i}}(t)}{t^{p-n+1}}\, dt \leq \frac{\alpha_{j}}{\alpha_{k+2}^{p-n+1}}.$$

As a consequence, we find that
$$\int_{\alpha_{k+2}} ^{\alpha_{k+1}} \sum_{j=k+2} ^{K} \sum_{i \in I_{j}} \frac{\mathbbm{1}_{T_{k,i}}(t)}{t^{p-n+1}}\, dt \leq \sum_{j=k+2} ^{K} \frac{e_j \alpha_{j}}{\alpha_{k+2}^{p-n+1}}.$$

Moreover, we have that
$$\int_{\alpha_{k+2}} ^{\alpha_{k+1}} \frac{\max\bigl\{S_k - e_{k+1} ,0\bigr\}^{\frac{p}{n-1}}}{S_k ^{\frac{p}{n-1}-1} t^{p-n+1}} \,dt = \frac{\max\bigl\{S_k - e_{k+1} ,0\bigr\}^{\frac{p}{n-1}}}{S_k ^{\frac{p}{n-1}-1} } \cdot \frac{\alpha_{k+1}^{n-p} -\alpha_{k+2} ^{n-p}}{n-p},$$
so \eqref{est:energy-2} becomes
\begin{multline*}
\hspace{-0.9em}\int_{\Omega} |\nabla u(x)|^p\,dx \geq (n-1)^{\frac{p}{2}} \omega_{n-1}\Biggl(\frac{\max\bigl\{S_k - e_{k+1} ,0\bigr\}^{\frac{p}{n-1}}}{S_k^{\frac{p}{n-1}-1} } \cdot \frac{\alpha_{k+1}^{n-p} -\alpha_{k+2} ^{n-p}}{n-p} - \frac{p}{n-1} \sum_{j=k+2} ^{K} \frac{e_j \alpha_{j}}{\alpha_{k+2}^{p-n+1}}\Biggr)\\
=(n-1)^{\frac{p}{2}} \omega_{n-1}\Biggl(\frac{\max\bigl\{S_k - e_{k+1} ,0\bigr\}^{\frac{p}{n-1}}}{S_k ^{\frac{p}{n-1}-1} } \cdot \frac{\alpha^{k+1}(1-\alpha)}{n-p} - \frac{p}{n-1} \sum_{j=k+2} ^{K} e_j \alpha^{\frac{j}{n-p}-\frac{(k+2)(p-n+1)}{n-p}}\Biggr),\hspace{-0.9em}
\end{multline*}
where in the last line we have just used \eqref{defn:alpha_k}.

Multiplying both sides by $(n-p)\alpha^k$ and summing over $k\in\{k_0,\dots,K\}$, for some integer $k_0\in \{0,\dots,K\}$, we obtain that
\begin{multline}
\frac{\alpha^{k_0}}{1-\alpha}(n-p)\int_{\Omega} |\nabla u(x)|^p\,dx \geq (n-1)^{\frac{p}{2}} \omega_{n-1} \Biggl(\sum_{k=k_0} ^{K} \frac{\max\bigl\{S_k - e_{k+1} ,0\bigr\}^{\frac{p}{n-1}}}{S_k ^{\frac{p}{n-1}-1}} \cdot \alpha^{2k+1}(1-\alpha)\\
- \frac{p(n-p)}{n-1} \sum_{k=k_0} ^{K} \sum_{j=k+2} ^{K} e_j \alpha^{k+\frac{j}{n-p}-\frac{(k+2)(p-n+1)}{n-p}}\Biggr),\label{est:energy-3}
\end{multline}
where the last sum is intended to be equal to zero when $k+2>K$, and we recall that $e_k=0$ when $k>K$.

Now we can compute the double sum as follows
\begin{align*}
\sum_{k=k_0} ^{K} \sum_{j=k+2} ^{K} e_j \alpha^{k+\frac{j}{n-p}-\frac{(k+2)(p-n+1)}{n-p}}&=\sum_{j=k_0 +2} ^{K} e_j \alpha^{\frac{j-2(p-n+1)}{n-p}} \sum_{k=k_0} ^{j-2} \alpha^{k\bigl(-\frac{1}{n-p} +2\bigr)}\\
&= \sum_{j=k_0 +2} ^{K} e_j \alpha^{\frac{j-2(p-n+1)}{n-p}} \frac{\alpha^{(j-1)\bigl(-\frac{1}{n-p}+2\bigr)} -\alpha^{k_0\bigl(-\frac{1}{n-p} +2\bigr)}}{\alpha^{-\frac{1}{n-p}+2}-1}\\
&\leq \sum_{j=k_0 +2} ^{K} e_j \alpha^{\frac{j-2(p-n+1)}{n-p}} \frac{\alpha^{(j-1)\left(-\frac{1}{n-p}+2\right)} }{\alpha^{-\frac{1}{n-p}+2}-1}\\
&= \frac{\alpha^{-\frac{1}{n-p}}}{\alpha^{-\frac{1}{n-p}+2}-1} \sum_{j=k_0 +2} ^{K} e_j \alpha^{2j},
\end{align*}
where the inequality holds because $\alpha^{2-1/(n-p)}>1$ thanks to \eqref{prop:alpha_p0}.

As for the first sum in \eqref{est:energy-3}, we apply Lemma~\ref{lemma:a_i} with
$$\beta=\frac{p}{n-1},\qquad \lambda=\alpha^2,\qquad a_0:= e_0+ \dots +e_{k_0}, \qquad a_k:=e_{k_0 +k} \quad \forall k>0, $$
so we find that
\begin{align*}
\sum_{k=k_0} ^{K} \frac{\max\bigl\{S_k - e_{k+1} ,0\bigr\}^{\frac{p}{n-1}}}{S_k ^{\frac{p}{n-1}-1} } \cdot \alpha^{2k+1}(1-\alpha) &\geq \alpha^{2k_0 +1}(1-\alpha) \cdot \frac{2\alpha^2-1}{2\alpha^2}\sum_{k=0} ^{K-k_0} a_k \alpha^{2k}\\
&\geq \frac{(1-\alpha)(2\alpha^2-1)}{2\alpha}\sum_{k=k_0} ^{K} e_k \alpha^{2k}.
\end{align*}

Plugging the last two estimates into \eqref{est:energy-3}, we get that
\begin{multline*}
\frac{\alpha^{k_0}}{1-\alpha} (n-p)\int_{\Omega} |\nabla u(x)|^p\,dx\\
\geq (n-1)^{\frac{p}{2}} \omega_{n-1} \biggl( \frac{(1-\alpha)(2\alpha^2-1)}{2\alpha} - \frac{p(n-p)}{n-1} \frac{\alpha^{-\frac{1}{n-p}}}{\alpha^{-\frac{1}{n-p}+2}-1} \biggr) \sum_{k=k_0} ^{K} e_k \alpha^{2k},
\end{multline*}
and hence, recalling \eqref{prop:alpha_p0}, we finally obtain that
\begin{equation}\label{est:final_ek}
\sum_{k=k_0} ^{K} e_k \alpha^{2k} \leq \frac{1}{(n-1)^{\frac{p}{2}} \omega_{n-1}} \cdot \frac{4\alpha}{(1-\alpha)^2 (2\alpha^2-1)}\cdot \alpha^{k_0} \cdot(n-p) \int_{\Omega} |\nabla u(x)|^p\,dx,
\end{equation}
for every $k_0\in\{0,\dots,K\}$.

Now we consider
\begin{gather*}
I_k ':=\{i\in \{1,\dots,N\}:z_i\in \Omega, |y_i-z_i|\in (\alpha_{k+1} ,\alpha_{k}]\},\qquad E_k ':=\bigcup_{i\in I_k} \{z_i\},\qquad e_k:=\mathcal{H}^0(I_k '),\\
E_0 ':=X ^-, \qquad e_0 ':=-X(E_0'),\qquad K':=\max\{k\in \N:e_k '>0\}, \qquad S_k':=\sum_{j=0} ^{k} e_j ',\\
D_k '(t):=\{x\in \Omega:\dist(x, E_0 ' \cup\dots\cup E_k ')< t\},
\end{gather*}
and we repeat the same argument used above, so we end up with
\begin{equation}\label{est:final_ek'}
\sum_{k=k_0} ^{K} e_k ' \alpha^{2k} \leq \frac{1}{(n-1)^{\frac{p}{2}} \omega_{n-1}} \cdot \frac{4\alpha}{(1-\alpha)^2 (2\alpha^2-1)}\cdot \alpha^{k_0} \cdot(n-p) \int_{\Omega} |\nabla u(x)|^p\,dx,
\end{equation}
for every $k_0\in\{0,\dots,K\}$.

We can now sum \eqref{est:final_ek} and \eqref{est:final_ek'}, so we get that
\begin{equation}\label{est:final_ek+ek'}
\sum_{k=k_0} ^{K} (e_k +e_k ') \alpha^{2k} \leq C(n,\alpha) \cdot \alpha^{k_0} \cdot(n-p) \int_{\Omega} |\nabla u(x)|^p\,dx,
\end{equation}
where
$$C(n,\alpha):=\frac{1}{(n-1)^{\frac{p}{2}} \omega_{n-1}} \cdot \frac{8\alpha}{(1-\alpha)^2 (2\alpha^2-1)}.$$

We are now ready conclude the first part of the proof. First, we notice that $\M_{\Omega}(X)=e_0+e_0'$, and hence
$$\M_{\Omega}(X)\leq C(n,\alpha) \cdot(n-p) \int_{\Omega} |\nabla u(x)|^p\,dx,$$
that is \eqref{th:mass}.

Then we observe that for every $i\in \{1,\dots, N\}$ we have that at least one between $y_i$ and $z_i$ lies in $\Omega$, and we recall that $|y_i-z_i|\leq \alpha_1$, so every $i$ belongs to some $I_k$ or $I_k'$. It follows that
$$\M_{\Omega}(S)\leq \sum_{i=1} ^{N} |y_i-z_i|\leq \sum_{k=1} ^{\max\{K,K'\}} \sum_{i\in I_k\cup I_k '} |y_i-z_i|\leq \sum_{k=1} ^{\max\{K,K'\}} (e_k+e_k')\alpha_{k}.$$

Recalling \eqref{defn:alpha_k}, the first property in \eqref{prop:alpha_p0}, and \eqref{est:final_ek+ek'}, we obtain that
\begin{multline*}
\M_{\Omega}(S)\leq \sum_{k=1} ^{\max\{K,K'\}} (e_k+e_k') \alpha^{\frac{k}{n-p}}= \alpha^{\frac{1}{2(n-p)}} \sum_{k=1} ^{\max\{K,K'\}} (e_k+e_k') \alpha^{\frac{k}{2(n-p)}}\alpha^{\frac{k-1}{2(n-p)}}\\
\leq \alpha^{\frac{1}{2(n-p)}} \sum_{k=1} ^{\max\{K,K'\}} (e_k+e_k') \alpha^{2k} \leq \alpha^{\frac{1}{2(n-p)}} \cdot C(n,\alpha)\alpha\cdot  (n-p) \int_{\Omega} |\nabla u(x)|^p\,dx,
\end{multline*}
that is \eqref{th:flat}.

We are left to prove \eqref{th:mass_sharp}. To this end, we go back to \eqref{est:energy-coarea} with $k=0$, but this time we do not neglect the energy in the set where $\dist(x,E_0)\in [0,\alpha_2]$, so we obtain that
$$\int_{\{\dist(x,E_0)<\alpha_1\}}|\nabla u(x)|^p\,dx\geq \int_{0} ^{\alpha_1} dt \int_{\partial D_0(t)} |\nabla u(x)|^p\,d\mathcal{H}^{n-1}(x),$$
and, similarly, also
$$\int_{\{\dist(x,E_0 ')<\alpha_1\}}|\nabla u(x)|^p\,dx\geq \int_{0} ^{\alpha_1} dt \int_{\partial D_0 '(t)} |\nabla u(x)|^p\,d\mathcal{H}^{n-1}(x).$$

Since the sets where $\dist(x,E_0)<\alpha_1$ and where $\dist(x,E_0')<\alpha_1$ are disjoint (and contained in $\Omega$) because of \eqref{eq:m(X_M)}, we also have that
\begin{equation}\label{est:coarea_sharp}
\int_{\Omega}|\nabla u(x)|^p\,dx\geq \int_{0} ^{\alpha_1} dt \int_{\partial D_0(t)} |\nabla u(x)|^p\,d\mathcal{H}^{n-1}(x) + \int_{0} ^{\alpha_1} dt \int_{\partial D_0 '(t)} |\nabla u(x)|^p\,d\mathcal{H}^{n-1}(x).
\end{equation}

Let us focus on the first addendum in the right-hand side, since the other one is similar. As before, we apply Lemma~\ref{lemma:W1p_boundary}, and we get that
$$\int_{0} ^{\alpha_1} dt \int_{\partial D_0(t)} |\nabla u(x)|^p\,d\mathcal{H}^{n-1}(x) \geq  n^{\frac{p}{n-1}} (n-1)^{\frac{p}{2}} \int_{0} ^{\alpha_1} \frac{\abs{\star Ju(D_0(t))}^{\frac{p}{n-1}}}{\mathcal{H}^{n-1}(\partial D_0(t))^{\frac{p}{n-1}-1}}\,dt.$$

Arguing as above, and more precisely using \eqref{est:Ju-1.5} and \eqref{est:meas_spheres}, we obtain that
$$\abs{\star Ju(D_0(t))}\geq \gamma_n e_0 - \gamma_n\sum_{j=1} ^{K} \sum_{i\in I_j} \mathbbm{1}_{T_{k,i}}(t),
\qquad\text{and}\qquad
\mathcal{H}^{n-1}(\partial D_0(t))\leq e_0 \omega_{n-1} t^{n-1},$$
and hence, exploiting the inequality $(a-b)^q\geq a^q - qa^{q-1}b$, we get
$$\frac{\abs{\star Ju(D_0(t))}^{\frac{p}{n-1}}}{\mathcal{H}^{n-1}(\partial D_0(t))^{\frac{p}{n-1}-1}}\geq \frac{\omega_{n-1}}{n^{\frac{p}{n-1}}} \Biggl(\frac{e_0}{t^{p-n+1}} - \frac{p}{n-1}\sum_{j=1} ^{K} \sum_{i\in I_j} \frac{\mathbbm{1}_{T_{k,i}}(t)}{t^{p-n+1}}\Biggr).$$

Now we estimate the negative contribution of the double sum in a different and more careful way than above. We observe that $T_{k,i}$ is not empty if and only if $\dist(z_i,E_0)<\dist(y_i,E_0)$, in which case $T_{k,i}=(\dist(z_i,E_0),\dist(y_i,E_0)]$. Moreover, the minimality property of $(y_i,z_i)$ and the fact that $E_0\subset (\star Ju)^+$, ensure that $|y_i-z_i|\leq \dist(z_i,E_0)$. Thus, it holds that
\begin{equation}\label{est:Tik-2}
\int_{0} ^{\alpha_1}\frac{\mathbbm{1}_{T_{k,i}}(t)}{t^{p-n+1}}\,dt \leq \int_{|y_i-z_i|} ^{2|y_i-z_i|} \frac{dt}{t^{p-n+1}}=\frac{2^{n-p}-1}{n-p}|y_i-z_i|^{n-p}.
\end{equation}

Putting together the last estimates, and recalling that $|y_i-z_i|\leq \alpha_j$ for every $i\in I_j$, we obtain that
\begin{align*}
(n-p)\int_{0} ^{\alpha_1} dt \int_{\partial D_0(t)} |\nabla u(x)|^p\,d\mathcal{H}^{n-1}(x) &\geq (n-1)^{\frac{p}{2}}\omega_{n-1} \Biggl(e_0 \alpha_1 ^{n-p} - \frac{p(2^{n-p}-1)}{n-1}\sum_{j=1} ^{K} e_j\alpha_j^{n-p} \Biggr)\\
&=(n-1)^{\frac{p}{2}}\omega_{n-1} \Biggl(e_0 \alpha - \frac{p(2^{n-p}-1)}{n-1}\sum_{j=1} ^{K} e_j\alpha^{j} \Biggr).
\end{align*}

With the same argument, we also obtain a similar estimate for the second addendum in the right-hand side of \eqref{est:coarea_sharp}, so we finally deduce that
\begin{align}
(n-p)\int_{\Omega} |\nabla u(x)|^p\,dx &\geq (n-1)^{\frac{p}{2}}\omega_{n-1} \Biggl((e_0+e_0') \alpha - \frac{p(2^{n-p}-1)}{n-1}\sum_{j=1} ^{\max\{K,K'\}} (e_j+e_j')\alpha^{j} \Biggr)\nonumber\\
&=(n-1)^{\frac{p}{2}}\omega_{n-1} \Biggl(\alpha\M_{\Omega}(X) - \frac{p(2^{n-p}-1)}{n-1}\sum_{j=1} ^{\max\{K,K'\}} (e_j+e_j')\alpha^{j} \Biggr).\qquad \label{est:sharp-1}
\end{align}

We still need to estimate the last sum. We point out that, at this point, we only have \eqref{est:final_ek+ek'}, which is not enough to deduce that this last sum is bounded by the $p$-energy of $u$.

In order to estimate this sum, we consider, for every $k\in\N$ the numbers
$$J_k:=\min\Bigl\{j\in\N: \alpha^j\leq (n-p)^{\frac{k}{2}}\Bigr\},
\quad\text{and}\quad
\widehat{K}:=\min\biggl\{k\in\N : (n-p)^{\frac{k}{2}} \leq \frac{1}{\M_\Omega (\star Ju/\gamma_n)}\biggr\},
$$
and we observe that
\begin{equation}\label{est:hatK}
\widehat{K}\leq 1+2\frac{\log\bigl(\M_\Omega(\star Ju/\gamma_n)\bigr)}{|\log (n-p)|}.
\end{equation}

Exploiting the definitions of $J_k$ and $\widehat{K}$ we obtain that
\begin{align}
\sum_{j=0} ^{\max\{K,K'\}} (e_j+e_j')\alpha^{j} &=\sum_{k=0} ^{\widehat{K}-1} \sum_{j=J_k} ^{J_{k+1}-1} (e_j+e_j')\alpha^{j}+ \sum_{j=J_{\widehat{K}}} ^{+\infty} (e_j+e_j')\alpha^{j}\nonumber\\
&\leq \sum_{k=0} ^{\widehat{K}-1} \sum_{j=J_k} ^{J_{k+1}-1} (e_j+e_j') \frac{\alpha^{2j}}{(n-p)^{\frac{k+1}{2}}} +\sum_{j=J_{\widehat{K}}} ^{+\infty} (e_j+e_j')\cdot (n-p)^{\frac{\widehat{K}}{2}}\nonumber\\
&\leq \sum_{k=0} ^{\widehat{K}-1} \frac{1}{(n-p)^{\frac{k+1}{2}}} \sum_{j=J_k} ^{+\infty} (e_j+e_j')\alpha^{2j} +\sum_{j=J_{\widehat{K}}} ^{+\infty} (e_j+e_j')\cdot\frac{1}{\M_{\Omega}(\star Ju/\gamma_n)},\qquad \label{est:sharp-2}
\end{align}
where we recall that, by definition, $e_j=0$ if $j>K$ and $e_j '=0$ if $j>K'$.

Now we estimate the first double sum using \eqref{est:final_ek+ek'} and \eqref{est:hatK}, so we get that
\begin{align*}
\sum_{k=0} ^{\widehat{K}-1} \frac{1}{(n-p)^{\frac{k+1}{2}}} \sum_{j=J_k} ^{+\infty} (e_j+e_j')\alpha^{2j}  &\leq \sum_{k=0} ^{\widehat{K}-1} \frac{\alpha^{J_k}}{(n-p)^{\frac{k+1}{2}}}\cdot C(n,\alpha) \cdot (n-p)\int_{\Omega} |\nabla u(x)|^p\,dx\\
&\leq \frac{\widehat{K}}{\sqrt{n-p}} \cdot C(n,\alpha) \cdot (n-p)\int_{\Omega} |\nabla u(x)|^p\,dx\\
&\leq \biggr(1+2\frac{\log(\M_\Omega(\star Ju/\gamma_n)}{|\log (n-p)|}\biggr)\cdot C(n,\alpha) \cdot \sqrt{n-p}\int_{\Omega} |\nabla u(x)|^p\,dx.
\end{align*}

Moreover, we clearly have that
$$\sum_{j=J_{\widehat{K}}} ^{+\infty} (e_j+e_j')\cdot\frac{1}{\M_{\Omega}(\star Ju/\gamma_n)}\leq \frac{1}{\M_{\Omega}(\star Ju/\gamma_n)} \sum_{j=0} ^{+\infty} (e_j+e_j') = 1.$$

Plugging the last two estimates into \eqref{est:sharp-2} we obtain that
$$\sum_{j=0} ^{\max\{K,K'\}} (e_j+e_j')\alpha^{j}\leq \biggr(1+2\frac{\log(\M_\Omega(\star Ju/\gamma_n)}{|\log (n-p)|}\biggr)\cdot C(n,\alpha) \cdot \sqrt{n-p}\int_{\Omega} |\nabla u(x)|^p\,dx +1,$$
and hence, from \eqref{est:sharp-1} we deduce that
\begin{multline*}
(n-p)\int_{\Omega} |\nabla u(x)|^p\,dx \geq (n-1)^{\frac{p}{2}}\omega_{n-1} \Biggl(\alpha\M_{\Omega}(X) - \frac{p(2^{n-p}-1)}{n-1}\\
-\frac{p(2^{n-p}-1)}{n-1} \biggr(1+2\frac{\log(\M_\Omega(\star Ju/\gamma_n)}{|\log (n-p)|}\biggr)\cdot C(n,\alpha) \cdot \sqrt{n-p}\int_{\Omega} |\nabla u(x)|^p\,dx \Biggr),
\end{multline*}
which implies \eqref{th:mass_sharp}, with
$$C=C(n,\alpha)\cdot n(n-1)^{\frac{p}{2}-1}\omega_{n-1}.$$
%
%


\end{proof}

\begin{rem}\label{rem:fixed_p}
We point out that the assumption that $u$ is smooth outside a finite set is not essential but it just allows for a more elementary proof. Indeed, since such maps are dense in $W^{1,p}(\Omega;\Sp^{n-1})$, one can recover \eqref{th:mass} and \eqref{th:flat} for a general map just by approximation. 

On the other hand, obtaining \eqref{th:mass_sharp} for $u\in W^{1,p}(\Omega;\Sp^{n-1})$ is more delicate, since in general the total mass of the Jacobian does not pass to the limit, so one would need to repeat the argument for $u\in W^{1,p}(\Omega;\Sp^{n-1})$ with $\M_\Omega(\star Ju)<+\infty$ (otherwise \eqref{th:mass_sharp} is trivial). However, this would just require a little generalization of Lemma~\ref{lemma:W1p_boundary}.

In particular, Theorem~\ref{teo:X_M+X_F} provides some estimates for the minimum problem \eqref{min_prescr_J} in the case $m=0$ also for a fixed $p\in (n-1,n)$, provided \eqref{prop:alpha_p0} holds.
\end{rem}

\begin{rem}
We discuss here a key point in the previous proof that might seem counterintuitive at first sight. Indeed, one could think that the second inequality in \eqref{est:energy-coarea} could be avoided and that it only worsens our estimates by a constant factor $1-\alpha$.

Actually, if we skipped that inequality, we would arrive at \eqref{est:energy-2} with the integral over $[0,\alpha_{k+1}]$ instead of $[\alpha_{k+2},\alpha_{k+1}]$. While this would yield a better estimate from below for the first term in the integral, in this case we would not be able to estimate the double sum involving the characteristic functions of the sets $T_{k,i}$, not even exploiting the improved description of these sets that we used afterwards to prove \eqref{th:mass_sharp}.

In fact, by reducing the integral to $t\in [\alpha_{k+2},\alpha_{k+1}]$, we obtain that the highest possible value for the integral on the intervals $T_{k,i}$ with $i\in I_j$ is achieved with $T_{k,i}=[\alpha_{k+2},\alpha_{k+2}+\alpha_j]$. On the other hand, if we kept the integral on $[0,\alpha_{k+1}]$, the worst case scenario would be $T_{k,i}=[\alpha_j , 2\alpha_{j}]$, as in \eqref{est:Tik-2}, which would result in a much worse estimate when $j$ is larger than $k+2$. For this reason, the second inequality in \eqref{est:energy-coarea} is crucial for the proof of \eqref{th:mass} and $\eqref{th:flat}$.

However, in order to obtain \eqref{th:mass_sharp} without a factor $1-\alpha$ in front of $\M_\Omega(X)$, we are forced to keep the integral in \eqref{est:coarea_sharp} on the whole $[0,\alpha_{1}]$, and hence to use \eqref{est:Tik-2} which produces the negative term in the right-hand side of \eqref{est:sharp-1}. Then we manage to estimate that term, but only exploiting \eqref{est:final_ek+ek'}, which was previously obtained by reducing the domain of integration and at the cost of allowing the appearance of the logarithm of the total mass, which in general cannot be bounded by the energy.
\end{rem}


We can now prove the compactness and $\Gamma$-liminf part of Theorem~\ref{main_res} in the case $m=0$.

\begin{proof}[Proof of $(i)$ in Theorem~\ref{main_res} when $m=0$]
Since $\Omega$ is a bounded Lipschitz domain, we can find an open neighborhood $U\subset \R^{n+m}$ of $\Omega$, so that any map $u\in W^{1,p}(\Omega;\Sp^{n-1})$ for some $p\in (n-1,n)$ can be extended to a map in $W^{1,p}(U,\Sp^{n-1})$, still denoted by $u$, in such a way that
$$\int_{U}|\nabla u|^p\,dx\leq C_{\Omega,U} \int_\Omega |\nabla u|^p\,dx,$$
for some constant $C_{\Omega,U}>0$ that does not depend on $p$ and $u$.

This can be done, for example, by reflection across $\partial \Omega$, as in \cite[Section~8]{2005-Indiana-ABO} or \cite[Section~15.6]{Bre-Mi}. In particular, \eqref{hp:energy-bound} holds also with $U$ instead of $\Omega$.

Hence, if we fix any $\delta \in (0,1/8)$ and any sequence $p_k\to n^-$ (and we set $u_k:=u_{p_k}$), we can apply Proposition~\ref{prop:approx}, so we obtain a new sequence of maps $\{\tilde{u}_k\}$ defined on $\Omega$, that are smooth outside a finite set (which depends on $k$), and satisfy \eqref{est:approx1}, \eqref{est:approx2} and \eqref{est:approx3}. In particular, it holds that
$$\limsup_{k\to +\infty}\, (n-\tilde{p}_k)\int_\Omega |\nabla \tilde{u}_k |^{\tilde{p}_k}\,dx<+\infty .$$

Now let us fix $\alpha \in (\sqrt{3}/2,1)$. For every $k$ which is sufficiently large that \eqref{prop:alpha_p0} holds with $p=\tilde{p}_k$, we can apply Theorem~\ref{teo:X_M+X_F} with $u=\tilde{u}_k$, so we find a sequence of integral $0$-currents $\{X_k\}$ and a sequence of integral $1$-currents $S_k$ for which $\star J\tilde{u}_k/\gamma_n= X_k+\partial S_k$ and such that \eqref{th:mass}, \eqref{th:flat} and \eqref{th:mass_sharp} hold with $p=\tilde{p}_k$.

In particular, $\M_\Omega (X_k)$ is uniformly bounded, and hence there exists a subsequence (not relabeled) and an integral $0$-current $\Sigma$ such that $\F_\Omega(X_k - \Sigma)\to 0$. Since we also know that $\M_\Omega(S_k)\to 0$, this implies that $\star J\tilde{u}_k\to \gamma_n \Sigma$ in the flat sense, and hence, thanks to Remark~\ref{rem:flat_approx}, we also obtain that $\star Ju_k\to \gamma_n \Sigma$.

As for the energies, we fix an open set $\Omega' \Subset \Omega$, and we exploit again Theorem~\ref{teo:X_M+X_F} in $\Omega'$, so we find sequence $\{X_k '\}$ and $\{S_k '\}$ (which do not necessarily coincide with the restriction of $\{X_k\}$ and $\{S_k\}$ to $\Omega'$) such that $\star J\tilde{u}_k/\gamma_n= X_k'+\partial S_k'$ in $\Omega'$, and $\M_{\Omega'}(S_k')\to 0$, so in particular $\F_{\Omega'}(X_k' -\Sigma)\to 0$.

Now we observe that \eqref{est:approx2} implies that
$$\limsup_{k\to +\infty} \frac{\M_\Omega(\star J\tilde{u}_k/\gamma_n)}{|\log (n-\tilde{p}_k)|}<+\infty,$$
so from \eqref{th:mass_sharp} we deduce that
$$ \liminf_{k\to +\infty} \, (n-\tilde{p}_k) \int_{\Omega'} |\nabla \tilde{u}_k|^{\tilde{p}_k}\,dx \geq \liminf_{k\to +\infty}(n-1)^{\frac{p_k}{2}}\omega_{n-1} \alpha \M_{\Omega'}(X_k ')\geq (n-1)^{\frac{n}{2}}\omega_{n-1} \alpha \M_{\Omega'}(\Sigma),$$
and therefore \eqref{est:approx3}, applied with $\Omega_1=\Omega'$ and $\Omega_2=\Omega$, yields
$$\liminf_{k\to +\infty} \, (n-p_k) \int_{\Omega} |\nabla u_k|^{p_k}\,dx \geq  (n-1)^{\frac{n}{2}}\omega_{n-1} \frac{\alpha}{L_\delta} \M_{\Omega'}(\Sigma).$$

Letting $\alpha\to 1^{-}$ and $\delta\to 0^{+}$, since this hold for every $\Omega'\Subset \Omega$, we obtain \eqref{eq: liminf main}.
\end{proof}

\section{Proof of the main result in general dimension and codimension}\label{sec: compactness and liminf general dim}

\subsection{Compactness and lower bound}\label{sbs: compact and lower bound}

In this section, we extend the proof of Theorem~\ref{main_res} to the general case in which $m>0$, by following the same strategy used in~\cite{2005-Indiana-ABO} for the Ginzburg-Landau functional.

The main idea is to deform the $m$-currents $\star Ju_p$ onto an $m$-dimensional grid in $\R^{n+m}$, in order to obtain a sequence of integral polyhedral $m$-currents that are close to the starting ones in the flat norm, and whose total mass can be estimated by looking only at the restriction of $u_p$ to a dual $n$-dimensional grid, so that we can apply our estimate for the case $m=0$.

The formalization of this argument is quite involved and requires the introduction of some notation. Following \cite[Section~3.2]{2005-Indiana-ABO}, for every $\ell>0$ and every $a \in \R^{n+m}$, let
$$\G(\ell,a):=\{a+z\ell+[0,\ell]^{n+m}:z\in \Z^{n+m}\}$$
denote the family of all cubes with sides of length $\ell$, whose vertices lie on the lattice $a+\ell\Z^{n+m}$.

We denote with $\G'(\ell,a):=\G(\ell,a+(\ell/2,\dots,\ell/2))$ the dual grid to $\G$, that is the grid consisting of the cubes whose vertices coincide with the centers of the cubes in $\G$.

For $h\in \{0,\dots,n+m\}$, we call $h$-cells of the grid $\G$ the $h$-dimensional faces of the cubes in $\G$, and for every $h$-cell $Q$ of $\G$, we denote by $Q'$ the unique $(n+m-h)$-cell of $\G'$ intersecting $Q$, and we call it the dual cell to $Q$ (notice that $Q\cap Q'$ consists of a single point that is the center of both $Q$ and $Q'$). Moreover, we denote by $R_h=R_h(\ell,a)$ the union of all the $h$-cells of $\G(\ell,a)$, which is called the $h$-skeleton of the grid $\G$. We denote by $R_h '$ the $h$-skeleton of $\G'$.

The following deformation lemma, which is a part of \cite[Lemma~3.8]{2005-Indiana-ABO}, is the main technical tool in the argument by Alberti, Baldo, and Orlandi.  

\begin{lemma}\label{lemma:deformation}
Let $\G=\G(\ell,a)$ be a grid and let $V\subset \R^{n+m}$ be a bounded open set. Let us set $U:=\{x\in \R^{n+m}: \dist(x,V)<2\ell\sqrt{n+m}\}$.

Then there exists a locally Lipschitz map $\Phi:\R^{n+m}\setminus R_{n-1}\to R_{m} '$, mapping every cube in $\G$ into itself, and such that the following statement holds.

For every smooth compact $m$-surface $M$ in $\R^{n+m}$ such that $\supp(M)\cap R_{n-1} = \partial M \cap U=\varnothing$, the push-forward $\Phi_\sharp M$ of the integral $m$-current $\lcurr M \rcurr$ is an integral polyhedral $m$-current supported on $R_m'$, that has no boundary in $V$, and such that
$$\F_V(\Phi_\sharp M - M)\leq C\ell^{m+1} \int_{M\cap U} \frac{d\mathcal{H}^{m}(x)}{\dist(x,R_{n-1})^{m}},$$
for some constant $C>0$, independent of $M,a,\ell$.

Moreover, for every $n$-cell $Q$ in $\G$ that is transversal to $M$, the multiplicity of $\Phi_\sharp M$ on the dual $m$-cell $Q'$ is constant and is equal (up to a sign) to the intersection number between $Q$ and $M$.
\end{lemma}

\begin{rem}\label{rem:Jacobian_intersection}
We point out that in the case in which $ \lcurr M \rcurr=\star Ju/\gamma_n$ for some map $u:\Omega\to \Sp^{n-1}$ that is smooth outside $M$, then the intersection number between $M$ and a transversal $n$-cell $Q$ is equal to $\star Ju_Q(Q)/\gamma_n$, where $u_Q$ denotes the restriction of $u$ to $Q$, whose Jacobian is an atomic measure on $Q$, supported on $M\cap Q$.
\end{rem}

Finally, we need the following two results, which are, respectively, minor modifications of \cite[Lemma~8.3]{2005-Indiana-ABO} and \cite[Lemma~8.4]{2005-Indiana-ABO}.

\begin{lemma}\label{lemma:int-dist-surf}
Let $S\subset \R^{d}$ be a bounded set that is contained in a finite union of Lipschitz surfaces with codimension $h$, and let us set $S_t:=\{x\in \R^d:\dist(x,S)<t\}$.

Then there exists a constant $C>0$, depending on $S$, such that for every $p<h$ and every $t>0$ it holds that
$$\int_{S_t} \frac{dx}{\dist(x,S)^p} \leq \frac{C}{h-p}t^{h-p}.$$
\end{lemma}

\begin{lemma}\label{lemma:elem}
Let $(X,\mathfrak{F},P)$ be a probability space and let $f_0,f_1,\dots,f_h:X\to [0,+\infty)$ be nonnegative measurable functions. For every $\delta>0$, let $A_\delta$ be the set of points $a\in X$ such that
$$f_0(a)\leq (1+\delta)\int_{X} f_0\,dP \qquad\text{and}\qquad f_i(a)\leq \frac{(1+\delta)2h}{\delta} \int_{X} f_i \,dP \quad \forall i\in\{1,\dots,N\}.$$
Then $P(A_\delta)\geq \delta/(2+2\delta)$.
\end{lemma}

The following proposition, which is analogue to \cite[Proposition~3.1]{2005-Indiana-ABO}, is essentially a localized version of the compactness and $\Gamma$-liminf part of Theorem~\ref{main_res}, for regular maps.

\begin{prop}\label{prop:cpt-liminf}
Let $\Omega_1\Subset \Omega_2 \subset \R^{n+m}$ be two bounded open sets with Lipschitz boundary. Let $\{p_k\}\subset (n-1,n)$ be a sequence of real numbers such that $p_k\to n^-$ and, for every $k\in \N$, let $\Sigma_k\subset \Omega_2$ be a smooth oriented $m$-submanifold without boundary and $u_k\in W^{1,p_k}(\Omega_2;\Sp^{n-1})\cap C^\infty(\Omega_2\setminus \Sigma_k;\Sp^{n-1})$ be a map with $\star Ju_k=\gamma_n \lcurr \Sigma_k \rcurr$. Let us assume that
\begin{equation}\label{hp:energy-bound-liminf}
\limsup_{k\to +\infty}\, (n-p_k)\int_{\Omega_2} |\nabla u_k|^{p_k} \,dx <+\infty,
\end{equation}
and that there exists $N>0$ for which
\begin{equation}\label{hp:mass_bound}
\limsup_{k\to +\infty} \frac{\M_{\Omega_2}(\star Ju_k)}{(n-p_k)^N}<+\infty.
\end{equation}

Then there exists a subsequence (not relabeled) and an integral $m$-boundary $\Sigma$ such that $\star Ju_k/\gamma_n \to \Sigma$ in the flat topology of $\Omega_1$, and, for every simple unit $m$-covector $\eta$ it holds that
\begin{equation}\label{eq:liminf_local}
\liminf_{k\to +\infty}\, (n-p_k)\int_{\Omega_2} |\nabla u_k|^{p_k} \,dx \geq (n-1)^{\frac{n}{2}}\omega_{n-1}|\Sigma \mres \eta|(\Omega_1),
\end{equation}
where $\Sigma \mres \eta$ is the $0$-current defined by $\langle\Sigma \mres \eta,\psi\rangle :=\langle \Sigma,\psi\eta\rangle$ for every $\psi \in \D^0 (\Omega_1)$.
\end{prop}

\begin{proof}


We can assume without loss of generality that $\eta=e_{n+1}\wedge\dots\wedge e_{n+m}$ and, up to passing to a subsequence, that the liminf in \eqref{eq:liminf_local} is a limit, so it does not change if we extract further subsequences.

Let us fix two real numbers $\alpha\in (\sqrt{3}/2,1)$ and $\delta\in (0,1)$, and let us set
$$\ell_k:=(n-p_k)^{N+1}  \qquad \forall k \in\N.$$

We claim that, for every $k\in\N$, there exists a set $A_k\subset [0,\ell_k]^{n+m}$ with measure larger than or equal to $\delta\ell_k^{n+m}/(2+2\delta)$ such that for every $a \in A_k$ it holds that $\Sigma_k$ is transversal to $R_n(\ell_k,a)$, disjoint from $R_{n-1}(\ell_k,a)$, and
\begin{align}
\ell_k ^{m}\int_{\widetilde{R}_{n}(\ell_k,a)\cap\Omega_2} |\nabla u_k|^{p_k}\,dx
&\leq (1+\delta)\int_{\Omega_2} |\nabla u_k|^{p_k}\,dx,\label{est:grid1}\\
\ell_k ^{m}\int_{R_{n}(\ell_k,a)\cap\Omega_2} |\nabla u_k|^{p_k}\,dx
&\leq \frac{C}{\delta}\int_{\Omega_2} |\nabla u_k|^{p_k}\,dx,\label{est:grid2}\\
\ell_k ^{m}\int_{\Sigma_k\cap \Omega_2} \frac{d\mathcal{H}^{m}(x)}{\dist(x,R_{n-1}(\ell_k,a))^{m}}
&\leq \frac{C}{\delta}\M_{\Omega_2}(\star Ju_k),\label{est:grid3}\\
\ell_k ^{m} \mathcal{H}^0(\Sigma_k\cap R_{n}(\ell_k,a)\cap\Omega_2)
&\leq \frac{C}{\delta} \M_{\Omega_2}(\star Ju_k),\label{est:grid4}
\end{align}
where $C>0$ is a constant depending only on $n$ and $m$, while $\widetilde{R}_{n}(\ell_k,a_k)$ is the union of all $n$-cells in $R_n(\ell_k,a)$ that are parallel to the $n$-plane spanned by $\{e_{1},\dots,e_{n}\}$.

Since the transversality and the disjointness conditions hold for almost every $a\in [0,\ell_k]^{n+m}$, the claim is a consequence of Lemma~\ref{lemma:elem} provided we check that
\begin{align}
\int_{[0,\ell_k]^{n+m}}\frac{da}{\ell_k ^n} \int_{\widetilde{R}_{n}(\ell_k,a)\cap\Omega_2} |\nabla u_k|^{p_k}\,dx
&\leq \int_{\Omega_2} |\nabla u_k|^{p_k}\,dx,\label{est:average1}\\
\int_{[0,\ell_k]^{n+m}}\frac{da}{\ell_k ^n} \int_{R_{n}(\ell_k,a)\cap\Omega_2} |\nabla u_k|^{p_k}\,dx
&\leq C \int_{\Omega_2} |\nabla u_k|^{p_k}\,dx,\label{est:average2}\\
\int_{[0,\ell_k]^{n+m}}\frac{da}{\ell_k ^n} \int_{\Sigma_k\cap \Omega_2} \frac{d\mathcal{H}^{m}(x)}{\dist(x,R_{n-1}(\ell_k,a))^{m}}
&\leq C\cdot \M_{\Omega_2}(\star Ju_k),\label{est:average3}\\
\int_{[0,\ell_k]^{n+m}}\frac{da}{\ell_k ^n} \,\mathcal{H}^0(\Sigma_k\cap R_{n}(\ell_k,a_k)\cap\Omega_2)
&\leq C\cdot \M_{\Omega_2}(\star Ju_k)\label{est:average4},
\end{align}
for some positive constant $C>0$ depending only on $n$ and $m$.

In order to prove these estimates, we first observe that for every non-negative measurable function $f:\Omega_2\to [0,+\infty)$ it holds that
\begin{equation}\label{eq:fubini}
\int_{[0,\ell_k]^{n+m}}\frac{da}{\ell_k ^n} \int_{\widetilde{R}_{n}(\ell_k,a)\cap\Omega_2} f(x)\,dx = \int_{[0,\ell_k]^{m}} da' \int_{\widetilde{R}_{n}(\ell_k,(0,\dots,0,a'_{1},\dots,a'_{m}))\cap\Omega_2} f(x)\,dx = \int_{\Omega_2} f(x)\,dx,
\end{equation}
where the first equality follows by the fact that $\widetilde{R}_{n}(\ell_k,a)$ does not depend on the first $n$ coordinates of $a$, and the second equality is Fubini's theorem.

Of course the same holds if we replace $\widetilde{R}_{n}(\ell_k,a)$ with the union of $n$-cells that are parallel to any other $n$-plane spanned by a set of $n$ vectors of the canonical basis of $\R^{n+m}$. Therefore we also have that
\begin{equation}\label{eq:binom}
\int_{[0,\ell_k]^{n+m}}\frac{da}{\ell_k ^n} \int_{R_{n}(\ell_k,a)\cap\Omega_2} f(x)\,dx = \binom{n+m}{n} \int_{\Omega_2} f(x)\,dx,
\end{equation}
and these two identities are enough to establish both \eqref{est:average1} and \eqref{est:average2}.

As for \eqref{est:average3}, we observe that
$$\int_{[0,\ell_k]^{n+m}}\frac{da}{\ell_k ^n} \int_{\Sigma_k\cap \Omega_2} \frac{d\mathcal{H}^{m}(x)}{\dist(x,R_{n-1}(\ell_k,a))^{m}}= \int_{\Sigma_k\cap \Omega_2}d\mathcal{H}^{m}(x)\int_{[0,\ell_k]^{n+m}}\frac{da}{\ell_k ^n \dist(x,R_{n-1}(\ell_k,a))^{m}},$$
and, for every $x\in \R^{n+m}$, it holds that
\begin{multline*}
\int_{[0,\ell_k]^{n+m}}\frac{da}{\ell_k ^n \dist(x,R_{n-1}(\ell_k,a))^{m}}=\int_{[0,\ell_k]^{n+m}}\frac{da}{\ell_k ^n \dist(0,R_{n-1}(\ell_k,a))^{m}}\\
=\int_{[0,1]^{n+m}}\frac{\ell_k ^m\,da}{ \dist(0,R_{n-1}(\ell_k,\ell_k a))^{m}}=\int_{[0,1]^{n+m}}\frac{da}{ \dist(0,R_{n-1}(1,a))^{m}},
\end{multline*}
and, thanks to Lemma~\ref{lemma:int-dist-surf}, the last integral is finite and clearly depends only on $n$ and $m$.

As a consequence, we have that
$$\int_{[0,\ell_k]^{n+m}}\frac{da}{\ell_k ^n} \int_{\Sigma_k\cap \Omega_2} \frac{d\mathcal{H}^{m}(x)}{\dist(x,R_{n-1}(\ell_k,a))^{m}}=\biggl(\int_{[0,1]^{n+m}}\frac{da}{ \dist(0,R_{n-1}(1,a))^{m}}\biggr) \mathcal{H}^m(\Sigma_k\cap \Omega_2),$$
which coincide, up to a dimensional constant, with $\M_{\Omega_2}(\star Ju_k)$, thus establishing \eqref{est:average3}.

Finally, we exploit the following slicing formula (see \cite[Theorem~3.2.22]{Federer-book})
$$\int_{\R^{m}} \mathcal{H}^0(\Sigma\cap(\R^{n}\times\{z\})\cap \Omega_2)\,dz \leq \mathcal{H}^m(\Sigma\cap \Omega_2),$$
which holds for every $m$-submanifold $\Sigma$, and hence, with a suitable adaptation of \eqref{eq:fubini} and \eqref{eq:binom} we obtain that
$$\int_{[0,\ell_k]^{n+m}}\frac{da}{\ell_k ^n} \,\mathcal{H}^0(\Sigma_k\cap R_{n}(\ell_k,a_k)\cap\Omega_2)\leq \binom{n+m}{n} \mathcal{H}^{m}(\Sigma_k\cap\Omega_2),$$
which yields \eqref{est:average4}.

\smallskip 

Now let us denote with $A_k '$ the set of all $a\in A_k$ such that
\begin{equation}
\frac{\alpha(n-1)^{\frac{p_k}{2}}\omega_{n-1}}{(1+\delta)^2}\sum_{Q\subset \widetilde{R}_n(\ell_k,a)\cap \Omega_2} \abs{\star J(u_k)_Q (Q)/\gamma_n}
\leq (n-p_k)\int_{\widetilde{R}_{n}(\ell_k,a)\cap\Omega_2} |\nabla u_k|^{p_k}\,dx, \label{est:mass_parallel}
\end{equation}
where the sum is intended to be over all $n$-cells $Q$ of $\G(\ell_k,a)$ that are contained in $\widetilde{R}_n(\ell_k,a)\cap \Omega_2$, and $\star J(u_k)_Q$ is the Jacobian of the restriction of $u_k$ to $Q$ (see Remark~\ref{rem:Jacobian_intersection}).

We claim that $|A_k'|\geq |A_k|-o(\ell_k ^{n+m})$, and in particular $A_k'$ has positive measure (and hence is not empty) when $k$ is large enough.


To prove this, we write every $a\in [0,\ell_k]^{n+m}$ as $a=(b,c)\in [0,\ell_k]^n\times [0,\ell_k]^m$, and for every $c\in [0,\ell_k]^m$, we set $B_{k}(c):=\{b\in [0,\ell_k]^{n}:(b,c)\in A_k 
\}$.
We also recall that the set $\widetilde{R}_n(\ell_k,a)=\widetilde{R}_n(\ell_k,(b,c))$, which is a union of $n$-planes parallel to $\operatorname{span}(e_1,\dots,e_n)$, depends only on $c$ and not on $b$, so we will denote it also with $\widetilde{R}_n(\ell_k,(\cdot,c))$, to emphasize the fact that this set is well-defined even if we do not fix $b$, namely the first component of $a$.

Now let us fix $c\in [0,\ell_k]^m$ such that $B_k(c)\neq \varnothing$, and let $P$ be an $n$-plane contained in $\widetilde{R}_n(\ell_k,(\cdot,c))$ and intersecting $\Omega_2$.
Let us set $\Omega_P:=\Omega_2\cap P$ and let $u_{k,P}$ denote the restriction of $u_k$ to $\Omega_P$. We point out that, since $B_k(c)\neq \varnothing$, fixing any $b\in B_k(c)$ we deduce that $P$ is transversal to $\Sigma_k$ and from \eqref{est:grid4} and \eqref{hp:mass_bound} we find that
\begin{equation}\label{est:mass_Jac_kP}
\M_{\Omega_P}(\star Ju_{k,P}) =\gamma_n \mathcal{H}^0(\Sigma_k\cap \Omega_P)\leq \frac{C}{\delta}\ell_k^{-m} \M_{\Omega_2}(\star Ju_k) \leq \frac{C}{\delta}\ell_k^{-m} (n-p_k)^{-N}
\end{equation}

Therefore, by Theorem~\ref{teo:X_M+X_F} there exist an integral $0$-current $X_{k,P}$ and an integral $1$-current $S_{k,P}$, made of a sum of segments, such that $\star Ju_{k,P}/\gamma_n = X_{k,P}+\partial S_{k,P}$ and \eqref{th:flat} yields
\begin{equation}\label{est:mass_SkP}
\M_{\Omega_P}(S_{k,P})\leq C_\alpha \alpha^{\frac{1}{2(n-p_k)}}(n-p_k) \int_{\Omega_P}|\nabla u_{k,P}(x)|^{p_k}\,dx\leq C_\alpha \alpha^{\frac{1}{2(n-p_k)}}(n-p_k) \int_{\Omega_P}|\nabla u_{k}(x)|^{p_k}\,dx,
\end{equation}
where $C_\alpha>0$ depends on $n$ and $\alpha$,
while, combining \eqref{est:mass_Jac_kP} and \eqref{th:mass_sharp}, we get
\begin{multline*}
\biggl(1+C_{\alpha}\frac{2^{n-p_k}-1}{\sqrt{n-p_k}}\biggl(1+ 2\frac{\log(C\delta^{-1}(n-p_k)^{-N}\ell_k^{-m})}{\abs{\log(n-p_k)}}\biggr) \biggr)(n-p_k) \int_{\Omega_P}|\nabla u_{k}(x)|^{p_k}\,dx\\
\geq (n-1)^{\frac{p_k}{2}}\omega_{n-1}\biggl(\alpha\M_{\Omega_P}(X_{k,P})-\frac{p_k(2^{n-p_k}-1)}{n-1}\biggr).
\end{multline*}

In particular, if $k$ is large enough (how large depends on $\delta,\alpha,N,n,m$) we have that
$$ (n-1)^{\frac{p_k}{2}}\omega_{n-1}\biggl(\alpha\M_{\Omega_P}(X_{k,P})-\frac{p_k(2^{n-p_k}-1)}{n-1}\biggr) \leq (1+\delta)(n-p_k) \int_{\Omega_P}|\nabla u_{k}(x)|^{p_k}\,dx,$$
and hence also
\begin{equation}\label{est:mass_Xkp}
\alpha (n-1)^{\frac{p_k}{2}}\omega_{n-1}\M_{\Omega_P}(X_{k,P})\leq (1+\delta)^2(n-p_k) \int_{\Omega_P}|\nabla u_{k}(x)|^{p_k}\,dx,
\end{equation}
since this inequality is trivial when $\M_{\Omega_P}(X_{k,P})=0$, while if $\M_{\Omega_P}(X_{k,P})\geq 1$ it follows from the previous one provided that $p_k(2^{n-p_k}-1)/(n-1)\leq \alpha \delta/(1+\delta)$, which is true when $k$ is large enough.

Now we observe that if $\supp S_{k,P} \cap R_{n-1}(\ell_k,a)=\varnothing$, then for every $n$-cell $Q$ in $\G(\ell_k,a)$ that is contained in $P$ we have that $\abs{\star J(u_k)_Q (Q)/\gamma_n}=|X_{k,P}(Q)|$ because each segment composing $S$ does not intersect $\partial Q\subset R_{n-1}(\ell_k,a)$, hence it is either contained in $Q$ or disjoint from $Q$, and in both cases $\partial S_{k,P}(Q)=0$. As a consequence, we have that
\begin{equation}\label{est:whenever}
\supp S_{k,P} \cap R_{n-1}(\ell_k,a)=\varnothing
\quad\Longrightarrow\quad\sum_{Q\subset \widetilde{R}_n(\ell_k,a)\cap P}\abs{\star J(u_k)_Q (Q)/\gamma_n} \leq \M_{\Omega_P}(X_{k,P}).
\end{equation}

Therefore, if $\supp S_{k,P} \cap R_{n-1}(\ell_k,a)=\varnothing$ for every $n$-plane $P$ contained in $\widetilde{R}_n(\ell_k,a)$ and intersecting $\Omega_2$, then combining \eqref{est:mass_Xkp} and \eqref{est:whenever} we obtain exactly \eqref{est:mass_parallel}, namely $a\in A_k '$.

This means that \eqref{est:mass_parallel} may fail for some $a\in A_k$ only if $\supp S_{k,P} \cap R_{n-1}(\ell_k,a)\neq \varnothing$ for some $n$-plane $P$ contained in $\widetilde{R}_n(\ell_k,a)$ and intersecting $\Omega_2$.

Now we remark that, while $\widetilde{R}_n(\ell_k,(b,c))$ depends only on $c$, the set $R_{n-1}(\ell_k,(b,c))\cap P$ do depend also on $b$, and actually the measure of the set of $b\in [0,\ell_k]^n$ for which $\supp S_{k,P} \cap R_{n-1}(\ell_k,(b,c))\neq \varnothing$ is less than or equal to $\sqrt{n}\cdot \M_{\Omega_P}(S_{k,P})$. This can be easily seen when $\supp S_{k,p}$ is a single segment, and then this estimate follows just by summing over all segments composing $S_{k,p}$.

As a consequence, recalling \eqref{est:mass_SkP}, we obtain that the measure of the set of $b\in B_{k}(c)$ for which $\supp S_{k,P} \cap R_{n-1}(\ell_k,a)\neq \varnothing$ for some $n$-plane $P$ contained in $\widetilde{R}_n(\ell_k,a)$ and intersecting $\Omega_2$ (so that \eqref{est:mass_parallel} might fail) is less than or equal to
$$\sqrt{n} C_\alpha \alpha^{\frac{1}{2(n-p_k)}}(n-p_k) \int_{\widetilde{R}_{n}(\ell_k,a)\cap \Omega_2} |\nabla u_{k,P}(x)|^{p_k}\,dx,$$
which, thanks to \eqref{est:grid1}, is bounded by $C \alpha^{\frac{1}{2(n-p_k)}}\ell_k ^{-m}$, where now the constant $C$ depends also on the bound \eqref{hp:energy-bound-liminf}. Thus, by Fubini's theorem, we obtain that
\begin{equation}\label{est:meas_A-A'}
\Leb^{n+m}(A_k\setminus A_k') \leq C \alpha^{\frac{1}{2(n-p_k)}}=o(\ell_k^{n+m}).
\end{equation}

In a similar way, we obtain that, for sufficiently large $k$, there is a subset $A_k''\subset A_k'$ with positive measure such that for every $a\in A_k''$ it also holds that
\begin{equation}
\sum_{Q\subset R_n(\ell_k,a)\cap \Omega_2} \abs{\star J(u_k)_Q (Q)/\gamma_n}
\leq C_\alpha (n-p_k)\int_{R_{n}(\ell_k,a)\cap \Omega_2} |\nabla u_k|^{p_k}\,dx,\label{est:mass_global}
\end{equation}
where the sum is intended to be over all $n$-cells $Q$ of $\G(\ell_k,a)$ that are contained in $R_n(\ell_k,a)\cap \Omega_2$, and $C_\alpha$ is a constant depending on $\alpha,n,m$.

Indeed, it is enough to repeat the previous argument for every possible choice of $n$ vectors of the canonical basis of $\R^{n+m}$, each time replacing $\widetilde{R}_n(\ell_k,a)$ with the union of the $n$-cells that are parallel to the $n$-plane spanned by those $n$ vectors, and writing $a=(b,c)\in [0,\ell_k]^n\times [0,\ell_k]^m$, where $b$ are the components of $a$ with respect to the $n$ chosen vectors and $c$ are the remaining ones. The only difference is that now we have to use \eqref{est:grid2} instead of \eqref{est:grid1}, so we obtain \eqref{est:meas_A-A'} with a constant depending also on $\delta$, which anyway is irrelevant. On the other hand, since we do not need to obtain sharp constants in \eqref{est:mass_global}, we could just use \eqref{th:mass} instead of \eqref{th:mass_sharp} to obtain the analogue of \eqref{est:mass_Xkp}.

In this way, for each different choice of $n$ vectors, we obtain that an estimate analogue to \eqref{est:mass_parallel} holds for every $a\in A_k '$, up to a small exceptional set whose measure is $o(\ell_k^{n+m})$. Finally, we sum all the $\binom{n+m}{m}$ inequalities obtained with all the possible choices of $n$ vectors, and we obtain \eqref{est:mass_global}.

We are now ready to conclude the proof. First of all, for every $k$ sufficiently large we take $a_k\in A_k''$ so that \eqref{est:mass_parallel}, \eqref{est:mass_global} and \eqref{est:grid1}-\eqref{est:grid4} hold with $a=a_k$.

Then we consider the maps $\Phi^k:\R^{n+m}\setminus R_{n-1}(\ell_k,a_k)\to R_m '(\ell_k,a_k)$ provided by Lemma~\ref{lemma:deformation}, and we set $\Sigma_k':=\Phi^k _\sharp(\Sigma_k)$.

If $k$ is sufficiently large so that $\dist(\Omega_1,\partial \Omega_2)>2\ell_k \sqrt{n+m}$, then from Lemma~\ref{lemma:deformation}, we deduce that $\Sigma_k'$ is an integral polyhedral $m$-current supported on $R_m '(\ell_k,a_k)$, with no boundary in $\Omega_1$, and such that
$$\F_{\Omega_1}(\Sigma_k'-\star Ju_k/\gamma_n)=\F_{\Omega_1}(\Sigma_k'- \Sigma_k)\leq C\ell_k ^{m+1} \int_{\Sigma_k\cap \Omega_2} \frac{d\mathcal{H}^{m}(x)}{\dist(x,R_{n-1}(\ell_k,a))^{m}},$$
hence \eqref{est:grid3} implies that
$$\F_{\Omega_1}(\Sigma_k'-\star Ju_k/\gamma_n)\leq \frac{C}{\delta}\ell_k \M_{\Omega_2}(\star Ju_k)\leq \frac{C}{\delta} \frac{\ell_k}{(n-p_k)^N},$$
and the right-hand side tends to zero as $k\to +\infty$, so the sequence $\{\star Ju_k\}$ is asymptotically equivalent to $\gamma_n \Sigma_k '$ in the flat topology of $\Omega_1$.

Moreover, from the second part of Lemma~\ref{lemma:deformation} and Remark~\ref{rem:Jacobian_intersection} we deduce that
$$\M_{\Omega_1}(\Sigma_k')\leq\sum_{Q\subset R_n(\ell_k,a_k)\cap \Omega_2}  \abs{\star J(u_k)_Q (Q)/\gamma_n}\ell_k ^m,$$
hence $\M_{\Omega_1}(\Sigma_k')$ is uniformly bounded due to \eqref{est:mass_global}, \eqref{est:grid2} and \eqref{hp:energy-bound-liminf}. Since $\M_{\Omega_1}(\partial \Sigma_k')=0$, by the Federer-Fleming compactness theorem we deduce that there exists a subsequence (not relabeled) and an integral $m$-current $\Sigma$ that is a boundary (since it is limit of boundaries) such that $\Sigma_k '\to \Sigma$, and hence also $\star Ju_k \to \gamma_n \Sigma$, in the flat topology of $\Omega_1$.

Moreover, we have that
$$|\Sigma\mres\eta|(\Omega_1)\leq \liminf_{k\to +\infty} |\Sigma_k '\mres \eta| (\Omega_1)\leq \liminf_{k\to +\infty} \sum_{Q\subset \widetilde{R}_n(\ell_k,a_k)\cap \Omega_2} \abs{\star J(u_k)_Q (Q)/\gamma_n}\ell_k ^m,$$
because, since $\eta = e_{n+1}\wedge\dots\wedge e_{n+m}$, the current $\Sigma_k '\mres \eta$ is just the restriction of $\Sigma_k'$ to the $m$-cells $Q'\subset R_m '(\ell_k,a_k)$ that are orthogonal to $\operatorname{span}(e_1,\dots,e_n)$, on which the multiplicity of $\Sigma_k'$ coincides with $\abs{\star J(u_k)_Q (Q)/\gamma_n}$ by Remark~\ref{rem:Jacobian_intersection}.

Hence, from \eqref{est:mass_parallel} and \eqref{est:grid1} we obtain that
\begin{align*}
(n-1)^{\frac{n}{2}}\omega_{n-1} |\Sigma\mres\eta|(\Omega_1) &\leq \frac{(1+\delta)^2}{\alpha}\liminf_{k\to +\infty} \,  \ell_k ^m (n-p_k)\int_{\widetilde{R}_{n}(\ell_k,a)\cap\Omega_2} |\nabla u_k|^{p_k}\,dx\\
&\leq \frac{(1+\delta)^3}{\alpha}\liminf_{k\to +\infty} \,  (n-p_k)\int_{\Omega_2} |\nabla u_k|^{p_k}\,dx.
\end{align*}

Letting $\delta\to 0^+$ and $\alpha\to 1^-$ we obtain \eqref{eq:liminf_local}.
\end{proof}

At this point, the proof of Theorem~\ref{main_res} follows from the combination of Proposition~\ref{prop:approx} and Proposition~\ref{prop:cpt-liminf}.

\begin{proof}[Proof of $(i)$ in Theorem \ref{main_res}, when $m>0$] 
Let us fix any sequence $p_k\to n^-$, and let us set $u_k:=u_{p_k}$. Up to a subsequence, not relabeled, we can assume that the liminf in \eqref{eq: liminf main} is a limit, so it does not change if we extract further subsequences.

As in the proof of the case $m=0$, we can extend each $u_k$ to some neighborhood $U$ of $\Omega$, so that
$$\int_{U} |\nabla u_k|^{p_k}\,dx \leq C_{\Omega,U} \int_{\Omega} |\nabla u_k|^{p_k} \,dx,$$
for every $k\in\N$, where $C_{\Omega,U}>0$ does not depend on $k$, so \eqref{hp:energy-bound} holds also with $U$ in place of $\Omega$.

Now we fix an open set $U'$ such that $\Omega\Subset U'\Subset U$, and $\delta \in (0,1/8)$ and we apply Proposition~\ref{prop:approx} with $\Omega=U'$, so we obtain a sequence $\{\tilde{p}_k\}$ such that $\tilde{p}_k\to n^-$ and a sequence of maps $\{\tilde{u}_k\}$ defined in $U'$, smooth outside smooth submanifolds, such that \eqref{est:approx2}, \eqref{est:approx3} and \eqref{approx_flat} hold with $\Omega=U'$.

In particular, \eqref{est:approx3} applied with $\Omega_1=U'$ and $\Omega_2=U$, and \eqref{hp:energy-bound} imply that
$$\limsup_{k\to +\infty} \, (n-\tilde{p}_k)\int_{U'} |\nabla \tilde{u}_k|^{\tilde{p}_k}\,dx <+\infty,$$
and this, together with \eqref{est:approx2}, ensures that the sequences $\{\tilde{p}_k\}$ and $\{\tilde{u}_k\}$ satisfy the assumptions of Proposition~\ref{prop:cpt-liminf} with $\Omega_1=\Omega$ and $\Omega_2=U'$.

Hence, from Proposition~\ref{prop:cpt-liminf} we deduce that, up to a subsequence, $\star J\tilde{u}_k \to \gamma_n \Sigma$, for some integral $m$-boundary $\Sigma$, and therefore, by \eqref{approx_flat}, also $\star Ju_k \to \gamma_n \Sigma$.

At this point, the liminf inequality follows from \eqref{est:approx3}, \eqref{eq:liminf_local} and a localization argument. To be precise, let us fix any finite family of open sets $A_1,\dots,A_N\Subset \Omega$ with pairwise disjoint closures, and some simple unit $m$-covectors $\eta_1,\dots,\eta_N$. Let also $A_1',\dots ,A_N '\Subset \Omega$ be pairwise disjoint open neighborhoods of $A_1,\dots,A_N$.

Then, combining \eqref{est:approx3}, applied with $\Omega_1=A_1'\cup\dots\cup A_N'$ and $\Omega_2 = \Omega$, with \eqref{eq:liminf_local} applied with $\Omega_1=A_i$, $\Omega_2=A_i'$ and $\eta=\eta_i$ for every $i\in \{1,\dots,N\}$, yields
\begin{align*}
\liminf_{k\to +\infty} \, (n- p_k) \int_{\Omega} |\nabla u_k|^{p_k}\,dx &\geq \sum_{i=1} ^{N} \liminf_{k\to +\infty} L_\delta ^{-1} (n-\tilde{p}_k) \int_{A_i '} |\nabla \tilde{u}_k|^{\tilde{p}_k}\,dx\\
&\geq L_\delta^{-1} (n-1)^{\frac{n}{2}} \omega_{n-1} \sum_{i=1} ^{N} |\Sigma\mres \eta_i|(A_i).
\end{align*}

Letting $\delta\to 0^+$ and taking the supremum over all admissible choices of $A_1,\dots,A_N$ and $\eta_1,\dots,\eta_N$, we obtain \eqref{eq: liminf main}.
\end{proof}

\subsection{Upper bound}\label{sbs: upper bound}

This section is devoted to the proof of $(ii)$ in Theorem \ref{main_res}, which provides the $\Gamma$-limsup inequality in the proof of the $\Gamma$-convergence. We first need the flat version of this estimate, that is the computation of the $p$-energy of the standard codimension $n$ vortex in $\R^{n+m}$. 

Let $ u_{n}^\star \colon (\R^n \setminus \{ 0 \}) \times \R^{m} \to \Sp^{n-1}$ be defined by
\begin{align}\label{eq: vortex def}
         u_{n}^\star (x', x'') := \frac{x'}{\abs{x'}} \qquad\forall \, (x',x'')\in (\R^n \setminus \{ 0 \}) \times \R^{m} .
\end{align}

\begin{lemma}\label{lem: flat vortex computation} Let $n\geq 2$, $m\geq 0$ be integers. Let $ A \subset \R^{m}$ be a bounded open set and $\delta>0$. Then, for every $p<n$ there holds 
\begin{equation*}
    \int_{B_\delta^n \times A} | \nabla u_{n}^\star|^p \, dx = (n-1)^{\frac{p}{2}} \omega_{n-1} \mathcal{H}^m(A) \frac{\delta^{n-p}}{n-p} . 
\end{equation*}
\end{lemma}

\begin{proof} By direct computation, we see that
\begin{equation*}
    \nabla u_{n}^\star (x',x'')= \frac{1}{|x'|} \left( \, {\textbf{I}}_{n\times n } - \frac{x'}{|x'|} \otimes \frac{x'}{|x'|} \, \bigg| \, { \textbf{0} }_{m\times n } \right) . 
\end{equation*}
Since $\textbf{I}_{n\times n } - \frac{x'}{|x'|} \otimes \frac{x'}{|x'|}$ is a projector onto an $(n-1)$-dimensional subspace in $\R^n$, it has $(n-1)$ eingevalues equal to $1$ and one eigenvalue equal to $0$, thus 
\begin{equation*}
    |\nabla u_{n}^\star| = \frac{1}{|x'|} \sqrt{n-1} . 
\end{equation*}

Then, by Fubini's theorem and polar coordinates 
\begin{align*}
    \int_{B_\delta^n \times A} | \nabla u_{n}^\star|^p \, dx & = \int_{A}  dx'' \int_{B^n _\delta} | \nabla u_{n}^\star|^p \, dx' \\ &= (n-1)^{\frac{p}{2}} \omega_{n-1} \mathcal{H}^m(A) \int_{0}^\delta \rho^{n-1-p} \, d\rho = (n-1)^{\frac{p}{2}} \omega_{n-1} \mathcal{H}^m(A) \frac{\delta^{n-p}}{n-p}.
\end{align*}
\end{proof}

The proof of the general result is based on a polyhedral approximation for integral currents and the existence of a map with a prescribed Jacobian on these polyhedral currents. There exist many variants of this kind of results in the literature; here we use two statements taken from \cite{2005-Indiana-ABO} and \cite{ABO2-singularities}, respectively, that we recall next. 

\begin{teo}[Proposition~8.6 in \cite{2005-Indiana-ABO}]\label{teo: poly approx} Let $n\geq 2$, $m\geq 0$ be integers, and let $\Omega\subset \R^{n+m}$ be a bounded Lipschitz domain. Let $\Sigma$ be an $m$-dimensional integral boundary in $\Omega$ with $\M_\Omega(\Sigma) <+\infty$. Then 
\begin{itemize}
    \item[$(i)$] $\Sigma=\partial N$ in $\Omega$, where $N$ is an $(m+1)$-dimensional integral current in $\R^{n+m}$ with compact support and $|\partial N|(\partial \Omega)=0$.
    \item[$(ii)$] There exists a sequence $\{N_i\}$ of $(m+1)$-dimensional polyhedral currents in $\R^{n+m}$ such that $\partial N_i \to \Sigma$ in $\F_\Omega$ and $\M_{\Omega}(\partial N_i) \to \M_{\Omega}(\Sigma)$. In addition, we can choose $N_i$ so that  $|\partial N_i|(\partial \Omega)=0$ and both $N_i$ and $\partial N_i$ have multiplicity $1$.
\end{itemize}

\end{teo}

\begin{teo}[Theorem~5.10 in \cite{ABO2-singularities}]\label{teo: sing competitor} Let $\Sigma =\partial N $ be an $m$-dimensional integral current with multiplicity $1$ which is the boundary of an $(m+1)$-polyhedral current $N$ in $\R^{n+m}$, and let $S_{m-1}^N$ denote the union of the faces of $N$ of dimension $m-1$. Then there exists $u\in W^{1, n-1}_{\rm loc}(\R^{n+m}; \Sp^{n-1})$ such that 
\begin{itemize}
    \item[$(i)$] $\star Ju = \gamma_n \Sigma $, $u$ is locally Lipschitz in $ \R^{n+m} \setminus (\Sigma \cup S_{m-1}^N)$ and is constant outside a bounded neighborhood of $N$. 
    \item[$(ii)$] $ |\nabla u| \in L^p(\R^{n+m})$ for every $p<n$ and satisfies $| \nabla u(x)| = O(1/{\rm dist}(x, \Sigma \cup S_{m-1}^N))$.
    \item[$(iii)$] For every $\gamma, \delta \in (0,1]$ small (depending only on $N$), for every $m$-dimensional face $F$ of $\Sigma$, if we identify the $m$-dimensional affine plane which contains $F$ with $\R^m$, and writing $ x =(x', x'') \in \R^{n} \times \R^m $, we have 
\begin{equation*}
    u(x) = \frac{x'}{|x'|}  , \s \mbox{for all } x\in \mathcal{U}_{\delta, \gamma}(F) , 
\end{equation*}
where 
\begin{equation}\label{eq: Ugammadelta def}
    \mathcal{U}_{\delta, \gamma }(F) = \big\{ (x',x'') \in \R^{n+m} \, : \, x''\in F, \, |x'| \le \min\{\delta, \gamma {\rm dist}(x'', \partial F) \} \big\} .
\end{equation}
Here $\delta, \gamma $ are small so that the analogous $\mathcal{U}_{\delta, \gamma }(\cdot)$ relative to the $m$-dimensional faces of $N$ have pairwise disjoint interiors (see the proof of \cite[Proposition 5.8]{ABO2-singularities} for more details on the choice of $\delta, \gamma$). In particular, for fixed $\gamma=\gamma(N)>0$ small, $\delta$ can be taken arbitrarily small.
\end{itemize}
\end{teo}

We now have all the ingredients to prove the existence of the recovery sequence for the upper bound inequality. 

\begin{proof}[Proof of $(ii)$ in Theorem \ref{main_res}] Fix an $m$-dimensional integral boundary $\Sigma$ in $\Omega$ and $p\in (n-1,n)$. 

First, by the polyhedral approximation Theorem \ref{teo: poly approx} and a diagonal argument, we can assume that $\Sigma$ has multiplicity $1$ and is the boundary (in $\Omega$) of an $(m+1)$-dimensional polyhedral current $N$ in $\R^{n+m}$ with multiplicity $1$ and $|\partial N|(\partial \Omega) = 0 $. In particular, $\Sigma$ is also polyhedral and satisfies $|\Sigma|(\partial \Omega)=0$. 

Denote by $S_{k}^\Sigma$ and $S_{k}^N$ the union of the $k$-dimensional faces of $\Sigma$ and $N$ respectively. Clearly, since $\partial N = \Sigma$ we have that $S_{k}^\Sigma \subset S_k^N$ for every $k\in \{0, \dotsc m\}$. Let $u$ be the map given by Theorem \ref{teo: sing competitor}, and consider $\delta, \gamma>0$ small as in $(iii)$ of Theorem \ref{teo: sing competitor}. Let also $r_\delta := \delta \sqrt{1+\gamma^{-2}}$. 

We cover $\Omega$ with three sets 
\begin{align*}
    \Omega \subset \Omega_F \cup \Omega_S \cup \Omega_E ,
\end{align*}
defined by 
\begin{equation*}
    \Omega_F  :=  \bigcup_{F\in S_m^\Sigma} \mathcal{U}_{\delta, \gamma}(F) \cap \Omega  ,
\end{equation*}
where $\mathcal{U}_{\delta, \gamma}(F)$ are defined in \eqref{eq: Ugammadelta def}, and 
\begin{equation*}
    \Omega_S  := \big(\{ {\rm dist}(\cdot, S_{m-1}^\Sigma) <r_\delta  \} \setminus \Omega_F \big) \cap \Omega , \qquad
    \Omega_E  :=   \{ {\rm dist} (\cdot, \Sigma) > \delta\} \cap \Omega. 
\end{equation*}
It is easily checked that the union of these three sets contains $\Omega$. Thus 
\begin{equation*}
    \int_{\Omega} |\nabla u|^p \, dx \le  \int_{\Omega_F} |\nabla u|^p \, dx+\int_{\Omega_S} |\nabla u|^p \, dx+\int_{\Omega_E} |\nabla u|^p \, dx , 
\end{equation*}
and we bound each contribution separately. 

By $(iii)$ of Theorem \ref{teo: sing competitor}, for every $F\in S_m^\Sigma$ we have that $u$ coincides up to a rigid motion with the standard vortex $u_n^\star$ (defined in \eqref{eq: vortex def}) in $\mathcal{U}_{\delta, \gamma}(F)$. With a little abuse of notation, we denote by $u_n^\star$ also the standard vortex composed with this rigid motion. Observe that, if $\Omega_\delta$ denotes a $\delta$-neighborhood of $\Omega$, we have  
\begin{equation*}
    \mathcal{U}_{\delta, \gamma}(F) \cap \Omega \subset B_\delta^n \times (F\cap \Omega_\delta) , \s \forall \, F\in S_m^\Sigma . 
\end{equation*}
Hence, for every $F\in S_m^\Sigma$, Lemma \ref{lem: flat vortex computation} gives
\begin{align*}
    \int_{\mathcal{U}_{\delta, \gamma}(F) \cap \Omega} |\nabla u|^p \, dx  = \int_{\mathcal{U}_{\delta, \gamma}(F) \cap \Omega} |\nabla u_n^\star|^p \, dx & \le \int_{B_\delta^n \times (F\cap \Omega_\delta)} |\nabla u_n^\star|^p \, dx \\ & =  (n-1)^{\frac{p}{2}} \omega_{n-1} \mathcal{H}^m(F\cap \Omega_\delta  ) \frac{\delta^{n-p}}{n-p} .
\end{align*} 
Adding up the contribution from all the faces, we obtain
\begin{equation}\label{eq: upper est faces}
    \int_{\Omega_F} |\nabla u|^p \, dx = \sum_{F\in S_m^\Sigma} \int_{ \mathcal{U}_{\delta, \gamma}(F) \cap \Omega } |\nabla u|^p \, dx \le (n-1)^{\frac{p}{2}} \omega_{n-1} \mathcal{H}^m(\Sigma\cap \Omega_\delta) \frac{\delta^{n-p}}{n-p} . 
\end{equation}

Now, we estimate the contribution on $\Omega_S$. By the pointwise estimate in $(ii)$ of Theorem \ref{teo: sing competitor} and the elementary inequality $1/\min\{a,b\} \le 1/a+1/b$ for $a,b>0$, we can bound
\begin{align*}
     \int_{\Omega_S} |\nabla u|^p \, dx & \le  C\int_{\Omega_S} \frac{dx}{ {\rm dist}(x, \Sigma \cup S^N_{m-1})^p} \\ &  \le C \int_{\Omega_S} \frac{dx}{ {\rm dist}(x, \Sigma)^p} + C \int_{\Omega_S} \frac{dx}{ {\rm dist}(x, S^N_{m-1})^p} . 
\end{align*}

Regarding the first term, since by definition $\Omega_S$ is disjoint from $\Omega_F$, we have that 
\begin{equation*}
    {\rm dist}(x, \Sigma) \ge \gamma {\rm dist}(x, S^N_{m-1}) ,  \, \s \forall \, x\in \Omega_S . 
\end{equation*}
Hence 
\begin{equation*}
    \int_{\Omega_S} |\nabla u|^p \, dx  \le C\biggl(\frac{1}{\gamma^p}+1 \biggr) \int_{\Omega_S} \frac{dx}{ {\rm dist}(x, S^N_{m-1})^p} . 
\end{equation*}
Moreover, by Lemma \ref{lemma:int-dist-surf} on the set $S^N_{m-1}$, which is a set of codimension $(n+1)$ in $\R^{n+m}$, we have that
\begin{equation*}
    \int_{\Omega_S} \frac{dx}{ {\rm dist}(x, S^N_{m-1})^p} \le  \int_{\{ {\rm dist}(\cdot, S^N_{m-1}) < r_\delta\}} \frac{dx}{ {\rm dist}(x, S^N_{m-1})^p} \le C_{N} \frac{r_{\hspace{-.8pt}\delta}^{n+1-p}}{n+1-p} , 
\end{equation*}
which gives
\begin{equation}\label{eq: upper est skeleton}
    \int_{\Omega_S} |\nabla u|^p \, dx \le C_N \biggl(\frac{1}{\gamma^p}+1 \biggr) \frac{r_{\hspace{-.8pt}\delta}^{n+1-p}}{n+1-p} 
\end{equation}

Lastly, we estimate the contribution of $\Omega_E$. By the very definition of $\Omega_E$ there holds 
\begin{equation*}
    {\rm dist}(x, \Sigma \cup S^N_{m-1}) \ge \min\{\delta, {\rm dist}(x, S^N_{m-1})\} , \s \forall \,  x\in \Omega_E .  
\end{equation*}
Thus, again by the estimate in $(ii)$ of Theorem \ref{teo: sing competitor}, we get 
\begin{align*}
      \int_{\Omega_E} |\nabla u|^p \, dx & \le  C\int_{\Omega_E} \frac{dx}{ {\rm dist}(x, \Sigma \cup S^N_{m-1})^p} \le C \int_{\Omega_E} \frac{dx}{ \min\{\delta, {\rm dist}(x, S^N_{m-1})\}^p} \\ & \le \frac{C}{\delta^p} \mathcal{H}^{n+m}(\Omega_E \cap \{ {\rm dist}(\cdot, S^N_{m-1}) >\delta\} ) + C\int_{\Omega_E \cap \{ {\rm dist}(\cdot, S^N_{m-1}) < \delta\}} \frac{dx}{ {\rm dist}(x, S^N_{m-1})^p} \\ & \le \frac{C}{\delta^p} \mathcal{H}^{n+m}(\Omega) + C\int_{\{ {\rm dist}(\cdot, S^N_{m-1}) < \delta\}} \frac{dx}{ {\rm dist}(x, S^N_{m-1})^p} .
\end{align*}
The second term in the last line can be estimated by Lemma \ref{lemma:int-dist-surf} as above, yielding 
\begin{equation}\label{eq: upper est exterior}
    \int_{\Omega_E} |\nabla u|^p \, dx \le \frac{C}{\delta^p} \mathcal{H}^{n+m}(\Omega) + C_N \frac{\delta^{n+1-p}}{n+1-p}. 
\end{equation}

Putting together \eqref{eq: upper est faces}, \eqref{eq: upper est skeleton}, and \eqref{eq: upper est exterior} gives
\begin{align*}
    \limsup_{p \to  n^-} \, (n-p) \int_{\Omega} |\nabla u|^p \, dx \le (n-1)^{\frac{n}{2}} \omega_{n-1} \mathcal{H}^m(\Sigma \cap \Omega_\delta),
\end{align*}
for every $\delta,\gamma \in (0,1]$ small as above. From here, to conclude, we observe that
\begin{equation*}
    \lim_{\delta \to 0^+} \mathcal{H}^m(\Sigma \cap \Omega_\delta) = |\Sigma|(\overline{\Omega}) = |\Sigma|(\Omega) = \M_{\Omega}(\Sigma) ,
\end{equation*}
since $|\Sigma|(\partial\Omega)=0$. Hence, letting $\delta \to 0^+$ for fixed $\gamma$ the conclusion follows. 
\end{proof}

\subsection{Boundary value problems}\label{sec: boundary val problems}

From Theorem~\ref{main_res}, following more or less the same strategy adopted in \cite{2005-Indiana-ABO} to pass from \cite[Theorem~1.1]{2005-Indiana-ABO} to \cite[Theorem~5.5]{2005-Indiana-ABO}, we can prove an analogous $\Gamma$-convergence result with Dirichlet boundary conditions.

Before stating the result, we recall that in the case $m\geq 1$, if $\Omega\subset \R^{n+m}$ is a bounded Lipschitz domain and $g\in W^{1-1/n,n}(\partial \Omega;\Sp^{n-1})$, then $\star Jg$ is a $(m-1)$-current on $\partial \Omega$, that is defined as $\partial (\star Ju)$ for any $u\in W^{1,n}(\Omega,\R^n)$ with trace $g$ on $\partial \Omega$ (see \cite{Hang-Lin-Jacobians} and \cite[Section~5.2]{2005-Indiana-ABO}). In what follows, for $p\in [1,n)$ we denote
\begin{equation*}
    W^{1,p}_g(\Omega; \R^n ) = \{ u\in W^{1,p}(\Omega; \R^n ) \, : \, u|_{\partial \Omega} =g \mbox{ in the sense of traces}\}.
\end{equation*}

Moreover, as in \cite[Section~5.1]{ABO2-singularities}, we say that two integer rectifiable $m$-currents $\Sigma_1$ and $\Sigma_2$ are \emph{cobordant} in $\overline{\Omega}$ if there exists an integral $(m+1)$-current $N$ supported in $\overline{\Omega}$ such that $\Sigma_1-\Sigma_2=\partial N$. Notice that this is equivalent to say $\partial \Sigma_1=\partial \Sigma_2$ and $\Sigma_1-\Sigma_2$ represent the trivial class in the $m$-th homology group of $\overline{\Omega}$.

\begin{teo}\label{thm:bvp}
Let $n\geq 2$, $m\geq 1$ be integers, and let $\Omega\subset \R^{n+m}$ be a bounded Lipschitz domain. Let $g\in W^{1-1/n,n}(\partial \Omega;\Sp^{n-1})$ be a map, fix $u\in W^{1,n}_g(\Omega;\R^{n})$, and let $\Sigma_y:=u^{-1}(y)$ be a regular level set of $u$ with $|y|<1$ (in the sense of \cite[Section~2.7]{2005-Indiana-ABO}).

\noindent Then the following statements hold.

\begin{itemize}
   \item[$(i)$](Compactness and $\Gamma$-liminf inequality) Let $p \in (n-1,n)$ be a (discrete) sequence with $p\to n^-$, and $u_p\in W^{1,p}_g(\Omega;\Sp^{n-1})$ be a sequence such that
\begin{equation*}
    \limsup_{p\to n^-} \, (n-p) \int_{\Omega}|\nabla u_p|^{p}\,dx <+\infty. 
\end{equation*}
    Then there exists a subsequence (not relabeled) and an integer rectifiable $m$-current $\Sigma$, cobordant with $\Sigma_y$ in $\overline{\Omega}$, such that, if we extend $\star Ju_p$ to zero outside $\Omega$, then $\F_{\R^{n+m}}(\star Ju_p/\gamma_n- \Sigma)\to 0$, and (along this subsequence) for every open set $A\subset \R^{n+m}$ it holds that
    \begin{equation*}
        \liminf_{p\to n^-} \, (n-p) \int_{\Omega \cap A }|\nabla u_p|^{p}\,dx \geq (n-1)^{\frac{n}{2}}\omega_{n-1} \M_A(\Sigma).
    \end{equation*}

    \item[$(ii)$] ($\Gamma$-limsup inequality) For every integer rectifiable $m$-current $\Sigma$, supported in $\overline{\Omega}$ and cobordant with $\Sigma_y$ in $\overline{\Omega}$, and every $p\in (n-1,n)$ there exist maps $u_p \in  W^{1,p}_g(\Omega,\Sp^{n-1})$ such that $\F_{\R^{n+m}}(\star Ju_p/\gamma_n -\Sigma)\to 0$ and
    $$\limsup_{p\to n^-} \, (n-p) \int_{\Omega}|\nabla u_p|^{p}\,dx\leq (n-1)^{\frac{n}{2}}\omega_{n-1} \M_{\R^{n+m}}(\Sigma).$$
\end{itemize}
\end{teo}

We do not include a detailed proof of this result since it can be obtained following the ideas in \cite[Sections~5-7]{2005-Indiana-ABO}. Instead, we just sketch the main steps of the proof.

As for the compactness and $\Gamma$-liminf, the idea is to extend $u$ to some neighborhood $U$ of $\overline{\Omega}$, so that $\|\nabla u\|_{L^n(U\setminus \overline{\Omega})}\leq \delta$, for some fixed small $\delta>0$. Then we extend also the maps $u_p$ to sphere-valued maps defined on $U$ by setting $u_{p}:=\pi_{a}(u)$ in $U\setminus\overline{\Omega}$, where $\pi_a$ is the projection defined in \eqref{defn:pi_a}. By suitably choosing $a\in B^n _{1/8}$, and possibly extracting a subsequence, we obtain that
$$(n-p)\int_{U\setminus\overline{\Omega}}|\nabla u_p|^{p}\,dx \leq C \delta,$$
for some dimensional constant $C>0$. At this point, the compactness with respect to $\F_U$ of $\star Ju_p$ and the liminf inequality follow by applying Theorem~\ref{main_res} in $U$.

With a little more effort (see \cite[Footnote~19]{2005-Indiana-ABO}) one can also pass from flat convergence in $U$ to flat convergence in $\R^{n+m}$. Finally, one can prove that the limit current is cobordant with $\Sigma_y$ using the stability of the cobordism with respect to the flat convergence and the following variant of \cite[Proposition~5.3]{2005-Indiana-ABO}.

\begin{lemma}
Let $\Omega\Subset U$ be bounded Lipschitz open sets, and let $u_1,u_2\in W^{1,n-1}(U;\Sp^{n-1})$ be maps such that $u_1=u_2$ in $U\setminus\overline{\Omega}$. Then $\star Ju_1/\gamma_n \mres \Omega$ and $\star Ju_2/\gamma_n \mres \Omega$ are cobordant in $\overline{\Omega}$.
\end{lemma}

\begin{proof}
Let $V$ be a Lipschitz open set such that $\Omega\Subset V\Subset U$, that admits a Lipschitz retraction onto $\Omega$. Let $\{\rho_\ep\}$ be a (discrete) sequence of mollifiers as in \eqref{defn:rho_ep}, and let $u_{1,\ep}=u_1*\rho_\ep$ and $u_{2,\ep}=u_2*\rho_\ep$. Then, for $\ep$ sufficiently small, we have that $u_{1,\ep}=u_{2,\ep}$ are defined in a neighborhood of $V$ and have the same trace on $\partial V$, and hence \cite[Proposition~5.3]{2005-Indiana-ABO} implies that $ \lcurr u_{1,\ep}^{-1}(y) \rcurr$ and $\lcurr u_{1,\ep}^{-2}(y) \rcurr$ are cobordant in $\overline{V}$ whenever $|y|<1$ is a regular value for both $u_{1,\ep}$ and $u_{2,\ep}$.

Now we choose $y$ to be a regular value for both $u_{1,\ep}$ and $u_{2,\ep}$ for every $\ep$ (in a countable sequence), and so that $\pi_y (u_{1,\ep})\to u_1$ and $\pi_y (u_{2,\ep})\to u_2$ in $W^{1,n-1}(V; \R^n)$. We remark that this holds for almost every $y$, up to extract a further subsequence of $\ep$.

As a consequence, we obtain that $\lcurr u_{1,\ep}^{-1}(y) \rcurr= \star J \pi_y (u_{1,\ep})/\gamma_n \to \star Ju_1/\gamma_n$ and, similarly, $\lcurr u_{2,\ep}^{-1}(y) \rcurr \to \star Ju_2/\gamma_n$, with respect to $\F_V$. Therefore, we deduce that $\star Ju_1/\gamma_n$ and $\star Ju_2/\gamma_n$ are cobordant in $\overline{V}$, namely $\star Ju_1/\gamma_n-\star Ju_2/\gamma_n= \star Ju_1/\gamma_n\mres \Omega -\star Ju_2/\gamma_n\mres \Omega$ is the boundary of an integral current that is supported in $\overline{V}$. By retracting this current into $\overline{\Omega}$, we find that $\star Ju_1/\gamma_n\mres \Omega $ and $\star Ju_2/\gamma_n \mres \Omega$ are cobordant also in $\overline{\Omega}$.
\end{proof}

As for the $\Gamma$-limsup inequality, following again \cite{2005-Indiana-ABO}, we can reduce ourselves to prove the statement along a subsequence and to consider currents with multiplicity $1$ that are polyhedral in some polyhedral domain $\Omega'\Subset \Omega$, and coincide with the Jacobian of a fixed map $u_a=\pi_a(u)$ in $\Omega\setminus \Omega'$, for a suitable choice of $a\in B_{1/8} ^n$. Then we can take the map provided by \cite[Theorem~9.6]{2005-Indiana-ABO} in $\Omega'$, and we extend it to $\Omega$ by setting it equal to $u_a$ in $\Omega\setminus \Omega'$.

The energy for this map in $\Omega'$ can be computed arguing exactly as in the proof of the $\Gamma$-limsup inequality without boundary conditions carried out in Subsection \ref{sbs: upper bound}. On the other hand, for a suitable choice of $a$ and of a subsequence, the limit of the normalized $p$-energy of $u_a$ in $\Omega\setminus\Omega'$ is small if $\Omega\setminus\Omega'$ is small, and this concludes the proof of the $\Gamma$-limsup inequality.

\smallskip

As an immediate consequence of Theorem~\ref{thm:bvp} we obtain the following.

\begin{cor}\label{cor: Hardt-Lin for currents} Let $n\geq 2$, $m\geq 1$ be integers, and let $\Omega\subset \R^{n+m}$ be a bounded Lipschitz domain. Let $g\in W^{1-1/n,n}(\partial \Omega;\Sp^{n-1})$ be a map, fix $u\in W^{1,n}(\Omega;\R^{n})$ with trace equal to $g$ on $\partial \Omega$, and let $\Sigma_y:=u^{-1}(y)$ be a regular level set of $u$ with $|y|<1$ (in the sense of \cite[Section~2.7]{2005-Indiana-ABO}). Let $u_p \in W^{1,p}(\Omega; \Sp^{n-1})$ be $p$-energy minimizing maps with $u_p|_{\partial \Omega} =g$ in the sense of traces, for $p\in (n-1,n)$.

Then, as $p \to  n^-$, a subsequence of $\star J u_p$ converges (in the flat topology of $\R^{n+m}$) to an integer rectifiable $m$-current $\Sigma$ that is area-minimizing among rectifiable $m$-currents cobordant with $\Sigma_y$ in $\overline{\Omega}$, and
\begin{equation}\label{th:limit-energy}
\lim_{p\to n^-} (n-p)\int_{\Omega} |\nabla u_p|^p\,dx = (n-1)^{\frac{n}{2}}\omega_{n-1} \M_{\R^{n+m}}(\Sigma).
\end{equation}
\end{cor}

Moreover, we also obtain a new proof of Theorem~\ref{thm: Hardt-Lin unpublished} as follows.

\begin{proof}[Proof of Theorem~\ref{thm: Hardt-Lin unpublished}]
First of all, we observe that Corollary~\ref{cor: Hardt-Lin for currents} implies that, up to a subsequence, $\star Ju_p\to \Sigma$, for some rectifiable $m$-current that is area-minimizing among rectifiable $m$-currents cobordant with $\Sigma_y$ in $\overline{\Omega}$.

Moreover, the total mass of $\mu_p$ is (up to a constant) the $p$-energy of the maps $u_p$, so \eqref{th:limit-energy} yields $\mu_p(\Omega)\to \M_{\R^{n+m}}(\Sigma)$ and hence, up to a further subsequence, we have that $\mu_p\rightharpoonup \mu$ for some finite measure $\mu$ on $\overline{\Omega}$ with $\mu(\overline{\Omega})=\M_{\R^{n+m}}(\Sigma)$.

Now we observe that the liminf estimate in Theorem~\ref{thm:bvp} implies that
$$\mu(\overline{A})\geq \limsup_{p\to n^-} \mu_p(A) \geq \liminf_{p\to n^-} \mu_p(A) \geq \M_A(\Sigma),$$
for every open set $A\subset\R^{n+m}$, and these two conditions are enough to infer that $\mu=|\Sigma|$.
\end{proof}

\section{Appendix} \label{sec:appendix}

We collect here the proofs of some technical results. The first one is Theorem~\ref{teo:BN-Invent}, whose proof is the same as in \cite{BN-Invent,Bre-Mi}, even if the statement is slightly more general.

\begin{proof}[Proof of Theorem~\ref{teo:BN-Invent}] The proof can be obtained following the same lines of \cite{BN-Invent}. The key observation is that $j u- j v$ differs only by an exact part from a sum of $(n-1)$-forms of the type $(u_i-v_i) X $, where $X$ is some exterior product of $(n-1)$ one-forms that are either $du_k$ or $dv_\ell$, for some $k, \ell \in \{1, \dotsc, n\}$. 

More precisely, we write 
\begin{equation}\label{eq:jacobian_sum_X}
u_i\widehat{du_i}-v_i\widehat{dv_i}=(u_i-v_i)X_{i,n}+v_i\sum_{j=1} ^{n-1} X_{i,j+1}-X_{i,j}, 
\end{equation}
where 
$$X_{i,j}:=\begin{cases}
du_1\wedge\dots du_{i-1}\wedge du_{i+1}\wedge \dots \wedge du_{j}\wedge dv_{j+1} \wedge \dots \wedge dv_n &\text{if }i<j,\\
du_1\wedge \dots \wedge du_{i-1}\wedge dv_{i+1}\wedge \dots \wedge dv_{n}&\text{if }i=j , \\
du_1\wedge \dots \wedge du_{j-1} \wedge dv_j\wedge \dots \wedge dv_{i-1}\wedge dv_{i+1}\wedge \dots \wedge dv_n &\text{if }i>j,
\end{cases}$$
with the convention that $X_{i,1}=\widehat{dv_i}$ and $X_{i,n}=\widehat{du_i}$, and obvious understanding when $i\in \{1,n\}$. 

Now we observe that if $i>j$, it holds that
\begin{align*}
X_{i,j+1}-X_{i,j}&=du_{1}\wedge \dots \wedge du_{j-1}\wedge (du_{j}-dv_{j})\wedge dv_{j+1} \wedge \dots \wedge dv_{i-1}\wedge dv_{i+1}\wedge \dots \wedge dv_n,\\
&=(-1)^{j-1} d\bigl((u_{j}-v_{j}) du_{1}\wedge \dots \wedge du_{j-1}\wedge dv_{j+1} \wedge \dots \wedge dv_{i-1}\wedge dv_{i+1}\wedge \dots \wedge dv_n\bigr),
\end{align*}
and hence
\begin{multline*}
d\bigl(v_i (X_{i,j+1}-X_{i,j})\bigr)=dv_i\wedge (X_{i,j+1}-X_{i,j})\\
=(-1)^{j} d\bigl((u_{j}-v_{j})dv_i\wedge du_{1}\wedge \dots \wedge du_{j-1}\wedge dv_{j+1} \wedge \dots \wedge dv_{i-1}\wedge dv_{i+1}\wedge \dots \wedge dv_n\bigr), 
\end{multline*}
which yields 
\begin{multline*}
\abs{\langle \star d\bigl(v_i (X_{i,j+1}-X_{i,j})\bigr), \psi\rangle}\\
=\biggl| \int_{\Omega} (u_{j}-v_{j})d\psi\wedge dv_i\wedge du_{1}\wedge \dots \wedge du_{j-1}\wedge dv_{j+1} \wedge \dots \wedge dv_{i-1}\wedge dv_{i+1}\wedge \dots \wedge dv_n \biggr|.
\end{multline*}

With an analogous argument, one obtains a similar formula also for $i=j$ or $i<j$, and also for the term $(u_i-v_i)X_{i,n}$ in \eqref{eq:jacobian_sum_X}. Hence, for $i>j$, by Hölder's inequality
\begin{align*}
\abs{\langle \star d\bigl(v_i (X_{i,j+1}-X_{i,j})\bigr), \psi\rangle} &\leq \|v_j-u_j\|_{L^q} \|\nabla u\|_{L^p} ^{j-1} \|\nabla v\|_{L^p} ^{n-j}\|d\psi\|_{L^\infty}\\
&\leq \|v-u\|_{L^q}\bigl( \|\nabla u\|_{L^p} + \|\nabla v\|_{L^p}\bigr)^{n-1}\|d\psi\|_{L^\infty},
\end{align*}
and similarly for all the other terms in \eqref{eq:jacobian_sum_X}. Combining all these estimates with \eqref{eq:jacobian_sum_X} and the definition of the Jacobian, we obtain \eqref{th:BN_invent}.
\end{proof}

Finally, we prove Lemma~\ref{lemma:a_i}, which is completely elementary.

\begin{proof}[Proof of Lemma~\ref{lemma:a_i}]
We argue by induction on $K$. If $K=0$ we have to check that
$$a_0\geq \frac{2\lambda-1}{2\lambda} a_0,$$
which is trivial.

So let us assume that the statement is true for some $K\in\N$ and let us prove that it also holds for $K+1$. To this end, we write the left-hand side of \eqref{th:lemma-a_i} with $K+1$ in place of $K$ as
\begin{align*}
&\sum_{k=0} ^{K+1} \frac{\max\{S_k-a_{k+1},0\}^{\beta}}{S_k ^{\beta-1}}\lambda^k\\
&\qquad= \sum_{k=0} ^{K-1} \frac{\max\{S_k-a_{k+1},0\}^{\beta}}{S_k ^{\beta-1}}\lambda^k + \frac{\max\{S_{K}-a_{K+1},0\}^{\beta}}{S_{K} ^{\beta-1}}\lambda^K + \frac{\max\{S_{K+1}-0,0\}^{\beta}}{S_{K+1} ^{\beta-1}}\lambda^{K+1}\\
&\qquad=\sum_{k=0} ^{K-1} \frac{\max\{S_k-a_{k+1},0\}^{\beta}}{S_k ^{\beta-1}}\lambda^k + S_{K} \lambda^K + \frac{\max\{S_{K}-a_{K+1},0\}^{\beta}}{S_{K} ^{\beta-1}}\lambda^K + S_{K+1}\lambda^{K+1} - S_{K}\lambda^K.
\end{align*}

The inductive hypothesis ensures that
$$\sum_{k=0} ^{K-1} \frac{\max\{S_k-a_{k+1},0\}^{\beta}}{S_k ^{\beta-1}}\lambda^k + S_{K}\lambda^{K}\geq \frac{2\lambda-1}{2\lambda} \sum_{k=0} ^{K} a_{k} \lambda^{k},$$
so we need to prove that
\begin{equation}\label{claim-induction}
\frac{\max\{S_{K}-a_{K+1},0\}^{\beta}}{S_{K} ^{\beta-1}}\lambda^K + S_{K+1}\lambda^{K+1} - S_{K}\lambda^K \geq \frac{2\lambda-1}{2\lambda} a_{K+1} \lambda^{K+1}.
\end{equation}

To this end we preliminarily observe that if $S_K=0$, then $S_{K+1}=a_{K+1}$ and \eqref{claim-induction} becomes
$$a_{K+1}\lambda^{K+1}\geq \frac{2\lambda-1}{2\lambda} a_{K+1} \lambda^{K+1},$$
which is trivial. Otherwise, we can set $t:=a_{K+1}/S_K$, hence $S_{K+1}=S_{K}+a_{K+1}=(1+t)S_{K}$, so \eqref{claim-induction} becomes
$$\frac{\max\{(1-t)S_{K},0\}^{\beta}}{S_{K} ^{\beta-1}}\lambda^K + (1+t)S_{K}\lambda^{K+1} - S_{K}\lambda^K \geq \frac{2\lambda-1}{2\lambda} t S_{K} \lambda^{K+1}.$$

As a consequence, it is enough to prove that
$$\max\{1-t,0\}^{\beta} + (1+t)\lambda - 1 \geq \frac{2\lambda-1}{2} t \qquad \forall t\geq 0.$$

Let us distinguish two cases, depending on whether the first addendum vanishes or not. If $t>1$, then we have to prove that
$$(1+t)\lambda - 1 -\frac{2\lambda-1}{2} t  \geq 0,$$
which is true since
$$(1+t)\lambda - 1 -\frac{2\lambda-1}{2} t = \lambda - 1 + t/2,$$
is positive for every $t>1$, because $\lambda\in (3/4,1)$.

If $t\in [0,1]$, then we have to prove that
$$(1-t)^\beta + (1+t)\lambda - 1 \geq \frac{2\lambda-1}{2} t,$$
which is equivalent to say that the function $\psi(t):=(1-t)^\beta + \lambda - 1 + t/2$ is nonnegative in $[0,1]$.

This is true because the minimum of $\psi$ is achieved at
$$t_0=1-\frac{1}{(2\beta)^{\frac{1}{\beta-1}}},$$
because $\psi$ is decreasing in $[0,t_0]$ and increasing in $[t_0,1]$, and hence
$$\psi(t)\geq \psi(t_0)=\frac{1}{(2\beta)^{\frac{\beta}{\beta-1}}}-\frac{1}{2(2\beta)^{\frac{1}{\beta-1}}} +\lambda-\frac{1}{2}
> 0 - \frac{1}{4} + \frac{3}{4}-\frac{1}{2}=0,$$
because $\beta \in (1,2)$ and $\lambda\in (3/4,1)$.
\end{proof}

\medskip \noindent
\textbf{Acknowledgements.}  
The authors are members of the {\selectlanguage{italian}``Gruppo Nazionale per l'Analisi Matematica, la Probabilità e le loro Applicazioni''} (GNAMPA) of the {\selectlanguage{italian}``Istituto Nazionale di Alta Matematica''} (INdAM).

M.~F. and N.~P. acknowledge the INdAM-GNAMPA Project {\selectlanguage{italian} \em Gamma-convergenza di funzionali geometrici non-locali}, CUP \#E5324001950001\#.

M.~F. is a member of the PRIN Project 2022AKNSE4 {\em Variational and Analytical aspects of Geometric PDEs}.

N.~P. acknowledges the MIUR Excellence Department Project awarded to the Department of Mathematics, University of Pisa, CUP I57G22000700001.

\begingroup

  \bibliographystyle{alpha-abbr.bst}

\bibliography{references}

\endgroup

\end{document}